\documentclass[11pt,article]{article}

\usepackage{amsmath}
\usepackage{amsthm}
  \usepackage{paralist}
  \usepackage{graphics} 
  \usepackage{epsfig} 
\usepackage{graphicx}
\usepackage{caption}
\usepackage{subcaption}
\usepackage{epstopdf}
 \usepackage[colorlinks=true]{hyperref}
 \usepackage{multirow}
\input{amssym.tex}

\newtheorem{theorem}{Theorem}[section]

\newtheorem{lemma}[theorem]{Lemma}
\newtheorem{proposition}{Proposition}

 \numberwithin{equation}{section}
\newtheorem{remark}{Remark}

\newcommand{\keywords}

\def\bc{\begin{center}}       \def\ec{\end{center}}
\def\ba{\begin{array}}        \def\ea{\end{array}}
\def\be{\begin{equation}}     \def\ee{\end{equation}}
\def\bea{\begin{eqnarray}}    \def\eea{\end{eqnarray}}
\def\beaa{\begin{eqnarray*}}  \def\eeaa{\end{eqnarray*}}

\def\mathbb{\Bbb}

\begin{document}

\title{\bf On a shadow system of the SKT competition system}
\author{Qi Wang \thanks{(Email:{\tt qwang@swufe.edu.cn}).  The author would like to thank the anonymous referee for his/her valuable comments and suggestions, which greatly improved the exposition of the paper.  This research is partially supported by the Fundamental Research Funds for the Central Universities, China.}\\
Department of Mathematics\\
Southwestern University of Finance and Economics\\
555 Liutai Ave, Wenjiang, Chengdu, Sichuan 611130, China
}

\date{}
\maketitle

%

%

%
%
%

%

\abstract
We study a boundary value problem with an integral constraint that arises from the modelings of species competition proposed by Lou and Ni in \cite{LN2}.  Through bifurcation theories, we obtain the existence of non-constant positive solutions of this problem, which are small perturbations from its positive constant solution, over a one-dimensional domain.  Moreover, we investigate the stability of these bifurcating solutions.  Finally, for the diffusion rate being sufficiently small, we construct infinitely many positive solutions with single transition layer, which is represented as an approximation of a step function.  The transition-layer solution can be used to model the segregation phenomenon through inter-specific competitions.

\textbf{Competition model, shadow system, nonlinear boundary value problem, transition layer.}

\section{Introduction}

In this paper, we consider the following one-dimensional nonlocal boundary value problem,
\begin{equation}\label{1}
\left\{
\begin{array}{ll}
\epsilon v''+(a_2-\frac{b_2 \lambda}{1+v}-c_2v)v=0,&x \in (0,L),     \\
v'(0)=v'(L)=0,\\
\int_0^L \frac{a_1-c_1v}{1+v}-\frac{b_1 \lambda}{(1+v)^2} dx=0,
\end{array}
\right.
\end{equation}
where $v=v_\epsilon(x)$ is a positive function and $\lambda=\lambda_\epsilon$ is a positive constant to be determined, while $ a_i, b_i, c_i$, $i=1,2$ and $\epsilon$ are some nonnegative constants.  

The motivation for studying model (\ref{1}) is that it is a limiting system or the so-called shadow system of the following Lotka-Volterra competition model with $\Omega=(0,L)$,
\begin{equation}\label{2}
\left\{
\begin{array}{ll}
u_t=\Delta[(d_1+\rho_{12}v)u]+(a_1-b_1u-c_1v)u, &x \in \Omega,~t>0,     \\
v_t=\Delta[(d_2+\rho_{21}u)v]+(a_2-b_2u-c_2v)v,& x \in \Omega,~t>0, \\
u(x,0)=u_0(x) \geq 0,~ v(x,0)=v_0(x) \geq 0,& x\in \Omega,\\
\frac{\partial u}{\partial \textbf{n}}=\frac{\partial v}{\partial \textbf{n}}=0,& x \in \partial \Omega,~t>0,
\end{array}
\right.
\end{equation}
where $d_1$, $d_2$ and $\rho_{12},\rho_{21}$ are positive constants, $d_i$ is referred as the diffusion rate and $\rho_{ij}$ as the cross-diffusion rate.  System (\ref{2}) was proposed by Shigesada et al. \cite{SKT} in 1979 to study the phenomenon of species segregation, where $u$ and $v$ represent the population densities of two competing species.  A tremendous amount of work has been done on the dynamics of its positive solutions since the proposal of system(\ref{2}).  There are also various interesting results on its stationary problem that admits non-constant positive solutions, in particular over a one-dimensional domain.  See \cite{E}, \cite{IMNY}, \cite{KY}, \cite{LN}, \cite{LN2}, \cite{LNW}, \cite{MEF}, \cite{MK}, \cite{MM}, \cite{MNTT}, and the references therein.

Great progress was made by Lou, Ni in \cite{LN, LN2} in the existence and quantitative analysis of the steady states of (\ref{2}) for $\Omega$ being a bounded domain in $\mathbb{R}^N$, $1\leq N\leq 3$.  Roughly speaking, they showed that (\ref{2}) admits only trivial steady states if one of the diffusion rates is large with the corresponding cross-diffusion fixed, and (\ref{2}) allows nonconstant positive steady states if one of the cross-diffusion pressures is large with the corresponding diffusion rate being appropriately given.  Moreover, they established the limiting profiles of non-constant positive solutions of (\ref{2}) as $\rho_{12} \rightarrow \infty$ (and similarly as $\rho_{21} \rightarrow \infty$).  For the sake of simplicity, we only state their results for $\rho_{21}=0$, while the same analysis can be carried out for the case when $\rho_{21}\neq0$.  Moreover, we refer our readers to \cite{WX} and the references therein for recent developments in the analysis of the shadow systems to (\ref{2}).  Suppose that $\frac{a_1}{a_2}\neq \frac{b_1}{b_2} \neq \frac{c_1}{c_2}$ and $d_2\neq a_2/\mu_j$ for any $j\geq1$, where $\mu_j$ is the $j-$th eigenvalue of $-\Delta$ subject to homogenous Neumann boundary condition.  Let $(u_i,v_i)$ be positive nonconstant steady states of (\ref{2}) with $(d_1,\rho_{12})=(d_{1,i},\rho_{12,i})$.  Suppose that $\rho_{12,i}/d_{1,i} \rightarrow r\in(0,\infty)$ as $\rho_{12,i} \rightarrow \infty$, then $(u_i,\rho_{12,i}v_i/d_{1,i}) \rightarrow (\lambda/(1+v),v)$ uniformly on $[0,L]$ for some positive constant $\lambda0$ and $v$ is a positive solution to the following problem,
\begin{equation}\label{3}
\left\{
\begin{array}{ll}
d_2 \Delta v+(a_2-b_2\lambda/(1+v)-(c_2/r)v)v=0,& x \in \Omega,\\
\frac{\partial u}{\partial \textbf{n}}=\frac{\partial v}{\partial \textbf{n}}=0,& x \in \partial \Omega,\\
\int_\Omega  (a_1-c_1 v)/(1+v)dx =b_1 \lambda \int_\Omega 1/(1+v)^2dx.
\end{array}
\right.
\end{equation}
We now denote $d_2=\epsilon$ since the smallness of diffusion rate tends to create nonconstant solutions for (\ref{3}).  Putting $c_2/r=\tilde{c}_2$ and assuming $\Omega=(0,L)$, we arrive at (\ref{1}), where we have dropped the tilde over $c_2$ in (\ref{1}) without causing any confusion.

For $\frac{a_1}{a_2}>\frac{b_1}{b_2}$ and if $\epsilon>0$ is small, Lou and Ni \cite{LN2} established the existence of positive solutions $v_\epsilon(\lambda_\epsilon,x)$ to (\ref{1}) by degree theory.  Moreover, $v_\epsilon(x)$ has a single boundary spike at $x=0$ if $\epsilon$ being sufficiently small.  This paper is devoted to study the solutions of (\ref{1}) that have a different structure, i.e, an interior transition layer.  The remaining part of this paper is organized as follows.  In Section 2, we carry out bifurcation analysis to establish nonconstant positive solutions to (\ref{1}) for all $\epsilon$ small.  The stability of these small amplitude solutions are then determined in Section 3 for $b_1=0$ in (\ref{1}).  In Section 4, we show that for any $x_0$ in a pre-determined subinterval of $(0,L)$, there exists positive solutions to (\ref{1}) that have a single interior transition layer at $x_0$.   Finally, we include discussions and propose some interesting questions in Section 5.

\section{Nonconstant positive solutions to the shadow system}
In this section, we establish the existence of nonconstant positive solutions to (\ref{1}).  First of all, we apply the following conventional notations as in \cite{LN,LN2}
\[A=\frac{a_1}{a_2},B=\frac{b_1}{b_2},C=\frac{c_1}{c_2},\]
then we see that (\ref{1}) has a constant solution
\[(\bar{v},\bar{\lambda})=\left(\frac{a_2}{c_2} \frac{B-A}{B-C}, \frac{a_2}{b_2}\frac{A-C}{B-C}\Big(1+\frac{a_2}{c_2} \frac{B-A}{B-C}\Big)\right),\]
and $\bar v, \bar \lambda>0$ if and only if
\begin{equation}\label{4}
B>A>C,~\text{or}~B<A<C.
\end{equation}
\subsection{Existence of positive bifurcating solutions}
To obtain non-constant positive solutions of (\ref{1}), we are going to apply the local bifurcation theory due to Crandall and Rabinowtiz \cite{CR}, therefore we shall assume (\ref{4}) from now on.  Taking $\epsilon$ as the bifurcation parameter, we rewrite (\ref{1}) in the abstract form
\[\mathcal{F}(v,\lambda,\epsilon)=0,~(v,\lambda,\epsilon) \in \mathcal{X}  \times \mathbb{R}^+\times \mathbb{R}^+,  \]
where
\begin{equation}\label{5}
\mathcal{F}(v,\lambda,\epsilon) =\left(
 \begin{array}{c}
\epsilon v''+(a_2-\frac{b_2 \lambda}{1+v}-c_2v)v\\
\int_0^L \frac{a_1-c_1v}{1+v} dx- \int_0^L \frac{b_1 \lambda}{(1+v)^2} dx
 \end{array}
 \right),
 \end{equation}
and $\mathcal{X}$ is a Hilbert space defined by
\[\mathcal{X}=\{w \in H^2(0,L) ~\vert~~ w'(0)=w'(L)=0 \}.\]
We first collect the following facts about the operator $\mathcal{F}$ before using the bifurcation theory.

\begin{lemma}
The operator $\mathcal{F}(v,\lambda,\epsilon)$ defined in (\ref{5}) satisfies the following properties:

(1)~$\mathcal{F}(\bar{v},\bar{\lambda},\epsilon)=0$ for any $\epsilon \in \mathbb{R}^+$;

(2)~$\mathcal{F}: \mathcal{X} \times \mathbb{R}^+ \times \mathbb{R}^+ \rightarrow \mathcal{Y} \times \mathcal{Y}$ is analytic, where $\mathcal{Y}=L^2(0,L)$;

(3)~for any fixed $(v_0,\lambda_0) \in \mathcal{X} \times \mathbb{R}^+$, the Fr\'echet derivative of $\mathcal{F}$ is given by

\begin{equation}\label{6}
 D_{(v,\lambda)}\mathcal{F}(v_0,\lambda_0,\epsilon)(v,\lambda) =\left(
 \begin{array}{c}
\epsilon  v''+\left(a_2-\frac{b_2\lambda_0}{(1+v_0)^2}-2c_2v_0 \right)v- \frac{b_2v_0}{1+v_0}\lambda\\
\int_0^L\left(\frac{2b_1\lambda_0}{(1+v_0)^3}-\frac{a_1+c_1}{(1+v_0)^2} \right)v- \frac{b_1 \lambda}{(1+v_0)^2} dx
 \end{array}
 \right);
\end{equation}

 (4)~$D_{(v,\lambda)}\mathcal{F}(v_0,\lambda_0,\epsilon): \mathcal{X} \times \mathbb{R}^+ \rightarrow \mathcal{Y} \times \mathbb{R}$ is a Fredholm operator with zero index.
 \end{lemma}

\begin{proof} Part (1)--(3) can be easily verified through direct calculations and we leave them to the reader.  To prove part (4), we formally decompose the derivative in (\ref{6}) as
\[ D_{(v,\lambda)}\mathcal{F}(v_0,\lambda_0,\epsilon)(v,\lambda)=D\mathcal{F}_1(v,\lambda)+D\mathcal{F}_2(v,\lambda),\]
where
\[D\mathcal{F}_1(v,\lambda)=\left(
 \begin{array}{c}
\epsilon  v''+\left( a_2-\frac{b_2\lambda_0}{(1+v_0)^2}-2c_2v_0 \right)v\\
0
 \end{array}
 \right),\]
and
\[D\mathcal{F}_2(v,\lambda)=\left(
 \begin{array}{c}
- \frac{b_2v_0}{1+v_0}\lambda\\
\int_0^L\left(\frac{2b_1\lambda_0}{(1+v_0)^3}-\frac{a_1+c_1}{(1+v_0)^2} \right)v- \frac{b_1 \lambda}{(1+v_0)^2} dx
 \end{array}
 \right).\]
Obviously $D\mathcal{F}_2: \mathcal{X} \times \mathbb{R}^+ \rightarrow \mathbb{R} \times \mathbb{R}$ is linear and compact.  On the other hand, $D\mathcal{F}_1$ is elliptic and according to Remark 2.5 of case 2, i.e, $N=1$, in Shi and Wang \cite{SW}, it is strongly elliptic and satisfies the Agmon's condition.  Furthermore, by Theorem 3.3 and Remark 3.4 of \cite{SW}, $D\mathcal{F}_1$ is a Fredholm operator with zero index.  Thus $D_{(v,\lambda)}\mathcal{F}(v_0,\lambda_0,\epsilon)$ is in the form of \emph{Fredholm operator+Compact operator}, and it follows from a well-known result, e.g, \cite{Ka}, that $D_{(v,\lambda)}\mathcal{F}(v_0,\lambda_0,\epsilon)$ is also a Fredholm operator with zero index.  Thus we have concluded the proof of this lemma.
\end{proof}

Putting $(v_0,\lambda_0)=(\bar{v},\bar{\lambda})$ in (\ref{6}), we have that
\begin{equation}\label{7}
 D_{(v,\lambda)}\mathcal{F}(\bar{v},\bar{\lambda},\epsilon)(v,\lambda) =\left(
 \begin{array}{c}
\epsilon  v''+\left(  \frac{a_2-c_2-2c_2 \bar{v}}{1+\bar{v}} \right)\bar{v}v- \frac{b_2\bar{v}}{1+\bar{v}}\lambda\\
\int_0^L\left( \frac{a_1-c_1-2c_1 \bar{v}}{(1+\bar{v})^2}  \right)v- \frac{b_1 \lambda}{(1+\bar{v})^2} dx
 \end{array}
 \right);
\end{equation}
To obtain candidates for bifurcation values, we need to check the following necessary condition on the null space of operator (\ref{7}),
\begin{equation}\label{8}
\mathcal{N}(D_{(v,\lambda)}\mathcal{F}(\bar{v},\bar{\lambda},\epsilon)) \neq \{ 0 \}.
\end{equation}
Let $(v,\lambda)$ be an element in this null space, then $(v,\lambda)$ satisfies the following system
\begin{equation}\label{9}
\left\{
\begin{array}{ll}
\epsilon  v''+\left( \frac{a_2-c_2-2c_2 \bar{v}}{1+\bar{v}} \right)\bar{v}v- \frac{b_2\bar{v}}{1+\bar{v}}\lambda=0,~x \in(0,L),\\
\int_0^L\left( \frac{a_1-c_1-2c_1 \bar{v}}{(1+\bar{v})^2}  \right)v- \frac{b_1 \lambda}{(1+\bar{v})^2} dx=0,\\
v'(0)=v'(L)=0.
\end{array}
\right.
\end{equation}
First of all, we claim that $\lambda=0$.  To this end, we integrate the first equation in (\ref{9}) over $0$ to $L$ and have that
\[ (a_2-c_2-2c_2 \bar{v} )\int_0^L v dx=b_2 \lambda L;\]
on the other hand, the second equation in (\ref{9}) is equivalent to
\[(a_1-c_1-2c_1 \bar{v}) \int_0^L v dx=b_1 \lambda L.\]
If $\lambda \neq 0$, we must have by equating the coefficients of the two identities above that
\[\bar{v}=\frac{B(a_2-c_2)-a_1+c_1}{2(B-C)c_2},\]
then by comparing this with the formula
\[\bar{v}=\frac{a_2}{c_2} \frac{B-A}{B-C},\]
we conclude from a straightforward calculation that $a_2(A-B)=c_2(B-C)$ and this implies that $\bar{v}=-1$ which is a contradiction.  Therefore $\lambda$ must be zero as claimed.

Now put $\lambda=0$ in (\ref{9}) and we arrive at
\begin{equation}\label{10}
\left\{
\begin{array}{ll}
\epsilon  v''+\left( \frac{a_2-c_2-2c_2 \bar{v}}{1+\bar{v}} \right)\bar{v}v=0,~x \in (0,L),\\
\left( \frac{a_1-c_1-2c_1 \bar{v}}{(1+\bar{v})^2} \right)\int_0^L v dx=0,~v'(0)=v'(L)=0.
\end{array}
\right.
\end{equation}
It is easy to see that (\ref{10}) has nonzero solutions if and only if $\frac{a_2-c_2-2c_2 \bar{v}}{(1+\bar{v})\epsilon}\bar{v}$ is one of the Neumann eigenvalues for $(0,L)$ and it gives rise to
\begin{equation}\label{11}
\frac{a_2-c_2-2c_2 \bar{v}}{(1+\bar{v})\epsilon}\bar{v}=(k\pi/L)^2,~ k \in N^+,
\end{equation}
which is coupled with an eigenfunction $v_k(x)=\cos(k\pi x/L)$.  Moreover we can easily see that the zero integral condition is obviously satisfied.
Then bifurcation might occur at $(\bar{v},\bar{\lambda},\epsilon_k)$ with
\begin{equation}\label{12}
\epsilon_k=\frac{a_2-c_2-2c_2 \bar{v}}{(1+\bar{v})(k\pi/L)^2}\bar{v},~ k \in N^+,
\end{equation}
provided that $\epsilon_k$ is positive or equivalently
\begin{equation}\label{13}
\bar{v}<\frac{a_2-c_2}{2c_2},~a_2>c_2.
\end{equation}
We have shown that the null space in (\ref{8}) is not trivial and in particular
\[\mathcal{N}\big(D_{(v,\lambda)}\mathcal{F}(\bar{v},\bar{\lambda},\epsilon)\big)=\text{span} \big\{\cos(k\pi x/L),0\big\}, k\in N^+.\]
\begin{remark}\label{rk2}
$(0,\frac{a_1}{b_1})$ is another trivial solution to (\ref{1}) and local bifurcation does not occur at $(0,\frac{a_1}{b_1})$.  Actually, putting $(v_0,\lambda_0)=(0,\frac{a_1}{b_1})$ in (\ref{6}), we see that the $v$-equation in (\ref{9}) becomes
\[\epsilon v''+(a_2-\frac{a_1b_2}{b_1})v=0,~v'(0)=v'(L)=0,\]
and the null-space $\mathcal{N}(D_{(v,\lambda)}\mathcal{F}(0,\frac{a_1}{b_1},\epsilon))$ must be trivial.
\end{remark}

Having the potential bifurcation values $\epsilon_k$ in (\ref{12}), we can now proceed to establish non-constant positive solutions for (\ref{1}) in the following theorem which guarantees that the local bifurcation occurs at $(\bar{v},\bar{\lambda},\epsilon_k)$.
\begin{theorem}\label{thm22} Assume that (\ref{4}) and (\ref{13}) hold.  Then for each $k\in N^+$, there exists $\delta>0$ and continuous functions
$s\in(-\delta, \delta):\rightarrow \epsilon_k(s) \in \mathbb{R}^+$ and
$s\in(-\delta, \delta):\rightarrow (v_k(s,x),\lambda_k(s)) \in \mathcal{X} \times \mathbb{R}^+$ such that
\begin{equation}\label{14}
\epsilon_k(0)=\epsilon_k,~(v_k(s,x),\lambda_k(s))=(\bar{v},\bar{\lambda})+s(\cos (k\pi x/L),0) +o(s),
\end{equation}
where $\epsilon_k$ is defined in (\ref{12}). Moreover, $(v_k(x,s),\lambda_k(s))$ solves system (\ref{1}) and all nontrivial solutions of (\ref{1}) near ($\bar{v},\bar{\lambda},\epsilon_k)$ are on the curve $\Gamma_k=(v_k(x,s),\lambda_k(s),\epsilon_k(s))$.
\end{theorem}
\begin{proof}
To make use of the local bifurcation theory of Crandall and Rabinowtiz \cite{CR}, we have verified all but the following \textit{transversality~condition}:
\begin{equation}\label{15}
\frac{d}{d \epsilon} \left(D_{(v,\lambda)}\mathcal{F}(\bar{v},\bar{\lambda},\epsilon)\right)(v_k,\lambda_k)\Big\vert_{\epsilon=\epsilon_k} \notin \mathcal{R}(D_{(v,\lambda)}\mathcal{F}(\bar{v},\bar{\lambda},\epsilon_k)),
\end{equation}
where
\[\frac{d}{d \epsilon} \left(D_{(v,\lambda)}\mathcal{F}(\bar{v},\bar{\lambda},\epsilon)\right)(v_k,\lambda_k)\Big\vert_{\epsilon=\epsilon_k}=\left(
 \begin{array}{c}
(\cos\frac{k\pi x}{L})''\\
0 \end{array}
 \right). \]
If not and (\ref{15}) fails, then there exists a nontrivial solution $(v,\lambda) \in \mathcal{X} \times \mathbb{R}^+$ that satisfies the following problem
\begin{equation}\label{16}
\left\{
\begin{array}{ll}
\epsilon  v''+\left( \frac{a_2-c_2-2c_2 \bar{v}}{1+\bar{v}} \right)\bar{v}v- \frac{b_2\bar{v}}{1+\bar{v}}\lambda=(\cos\frac{k\pi x}{L})'',~x\in (0,L),\\
\int_0^L\left( \frac{a_1-c_1-2c_1 \bar{v}}{(1+\bar{v})^2}  \right)v- \frac{b_1 \lambda}{(1+\bar{v})^2} dx=0,\\
v'(0)=v'(L)=0.
\end{array}
\right.
\end{equation}
By the same analysis that leads to the claim below (\ref{9}), we can show that $\lambda=0$ in (\ref{16}), which then becomes
\begin{equation}\label{17}
\left\{
\begin{array}{ll}
\epsilon_k v''+\left( \frac{a_2-c_2-2c_2 \bar{v}}{1+\bar{v}} \right)\bar{v}v=(\cos\frac{k\pi x}{L})'',~ x\in (0,L),\\
v'(0)=v'(L)=0.
\end{array}
\right.
\end{equation}
However, this reaches a contradiction to the Fredholm Alternative since $\cos\frac{k\pi x}{L}$ is in the kernel of the operator on the left hand side of (\ref{17}).  Hence we have proved the transversality condition and this concludes the proof of Theorem \ref{thm22}.
\end{proof}

\subsection{Global bifurcation analysis}

We now proceed to extend the local bifurcation curves obtained in Theorem \ref{thm22} by the global bifurcation theory of Rabinowitz in its version developed by Shi and Wang in \cite{SW}.  In particular, we shall study the first bifurcation branch $\Gamma_1$.

\begin{theorem}\label{thm31}Assume that (\ref{4}) and (\ref{13}) hold.  Then there exists a component $\mathcal{S} \subset \mathcal{X}  \times  \mathbb{R}^+ \times  \mathbb{R}^+$ that satisfies,

(i) $\mathcal{S}$ contains $(v_1(x,s),\lambda_1(s),\epsilon_1(s)), s \in (-\delta, \delta)$;

(ii) $\forall (v,\lambda,\epsilon) \in  \mathcal{S}$, $v(x)>0$ on $[0,L]$, $\lambda>0$ and $(v,\lambda)$ is a solution of (\ref{1});

(iii) $\mathcal{S}=\mathcal{S}_u \cup\mathcal{S}_l$ with $\mathcal{S}_u \cap \mathcal{S}_l=(\bar{v},\bar{\lambda},\epsilon_1)$ such that for all small $\epsilon >0$, $\mathcal{S}_u\backslash(\bar{v},\bar{\lambda},\epsilon_1)$ consists of $(v,\lambda,\epsilon)$ with $v'(x)>0$ on $(0,L)$ and $\mathcal{S}_l \backslash(\bar{v},\bar{\lambda},\epsilon_1)$ consists of $(v,\lambda,\epsilon)$ with $v'(x)<0$ on $(0,L)$, where
\[\epsilon_1=\frac{a_2-c_2-2c_2\bar v}{(1+\bar v)(\pi/L)^2};\]

(iv) $\forall \epsilon \in (0,\epsilon_1)$, there exists $(v,\lambda,\epsilon) \in \mathcal{S}_u$ and the same holds for $\mathcal{S}_l$.

\end{theorem}

\begin{proof}
Denote the solution set of (\ref{1}) by
\[\mathcal{D}=\{(v,\lambda,\epsilon) \in \mathcal{X} \times \mathbb{R}^+ \times \mathbb{R}^+ ~\vert~ \mathcal{F}(v,\lambda,\epsilon)=0, (v,\lambda) \neq (\bar{v},\bar{\lambda})\}\] and choose $\mathcal{S}$ to be the maximal connected subset of $\bar{\mathcal{D}}$ that contains $(\bar{v},\bar{\lambda},\epsilon_1)$.  Then $\mathcal{S}$ is the desired closed set and $(i)$ follows directly from (\ref{14}) in Theorem \ref{thm22}.

To prove that $v(x)$ is positive on $[0,L]$ and $\lambda$ is positive for all $(v,\lambda,\epsilon)\in\Psi$ with $\epsilon>0$, we introduce the following two connected sets:
\[\mathcal{S}^+=\mathcal{S}\cap \{\epsilon>0\},\]
and
\[\mathcal{P}_0=\{(v,\lambda) \in \mathcal{S}^+ ~\vert~ v(x)>0,~x\in[0,L],~\lambda>0\},\]
then we want to show that $\mathcal{P}_0=\mathcal{S}^+$.  First of all, we observe that $\mathcal{P}_0$ is a subset of $\mathcal{S}^+$ and $\mathcal{P}_0$ is nonempty, since at least the part of $\mathcal{S}$ near $(\bar{v},\bar{\lambda},\epsilon_1)$ is contained in $\mathcal{P}_0$.  Now we prove that $\mathcal{P}_0$ is both open and closed in $\mathcal{S}^+$.  The openness is trivial, since for any $(v,\lambda,\epsilon) \in \mathcal{P}_0$ and the sequence $(v_k,\lambda_k,\epsilon_k)$ that converges in $(v,\lambda,\epsilon)$ in $\mathcal{X}\times \mathbb{R}^+ \times\mathbb{R}^+$, we must have that $v_k$ converges to $v$ in $C^2([0,L])$, therefore $v_k>0$ on $[0,L]$ since $v>0$ on $[0,L]$.  Furthermore, the fact that $\lambda_k>0$ and $\epsilon_k>0$ follows readily from $\lambda>0$ and $\epsilon>0$.

Now we show that $\mathcal{P}_0$ is closed in $\mathcal{S}^+$.  Take $\{(v_k,\lambda_k,\epsilon_k)\} \in \mathcal{P}_0$ such that $(v_k,\lambda_k,\epsilon_k) \rightarrow (v,\lambda,\epsilon) \text{ in the norm of }\mathcal{X}\times \mathbb{R}^+ \times\mathbb{R}^+$ for some $(v,\lambda,\epsilon) \in \mathcal{S}^+.$  We want to show that $(v,\lambda,\epsilon) \in \mathcal{P}_0$, i.e $v(x)>0$ on $[0,L]$ and $\lambda>0$.  Obviously we have $v(x)\geq0$ and $\lambda\geq0$.  Now we show that $\lambda\neq 0$ and $v(x)>0$ for all $x_0\in [0,L]$.  We argue by contradiction.

If $\lambda=0$, the $v(x)$--equation in (\ref{1}) becomes
\begin{equation}\label{s12}
\left\{
\begin{array}{ll}
\epsilon v''+(a_2-c_2v)v=0, x\in (0,L),\\
v'(0)=v'(L)=0.
\end{array}
\right.
\end{equation}
It is well-known that (\ref{s12}) has only trivial solution, i.e, $v\equiv0$, or $v\equiv \frac{a_2}{c_2}$, hence $v_k$ converges to either 0 or $\frac{a_2}{c_2}$ uniformly on $[0,L]$.  The case that $v_k$ converges to 0 can be treated by the same analysis that shows $v_k(x)>0$.  If $v_k$ converges to $\frac{a_2}{c_2}$, we apply the Lebesgue's Dominated Convergence Theorem to the integral constraint in (\ref{1}) and send $\lambda_k \rightarrow 0$, then we have that
\[0=\lim_{k\rightarrow \infty} \int_0^L \frac{1}{1+v_k}\left(a_1-\frac{b_1\lambda_k}{1+v_k}-c_1v_k\right) dx=\frac{a_2c_2L}{a_2+c_2}(A-C),\]
and it implies that $A=C$, however this is a contradiction to (\ref{4}).  Therefore $\lambda$ must be positive as desired.  On the other hand, it is easy to see that $v(x)\geq0$ on $[0,L]$ and if $v(x_0)=0$ for some $x_0\in[0,L]$.  We apply the Strong Maximum principle and Hopf's lemma to (\ref{1}) and have that $v\equiv 0$ for all $x\in [0,L]$ and $\lambda=\frac{a_1}{b_1}$.  However, we have from Remark \ref{rk2} that bifurcation does not occur at $(0,\frac{a_1}{b_1})$.  This is a contraction and we must have that $v(x)>0$ on $[0,L]$.

To prove $(iii)$, we choose $\mathcal{S}_u$ to be the component of $\mathcal{S} \backslash \{(v_1(s,x),\lambda_1(s),\epsilon_1(s)) \vert s \in(-\delta,0) \}$ containing $ \{(v_1(s,x),\lambda_1(s),\epsilon_1(s)) \vert s\in [0 ,\delta)  \}$
and correspondingly $\mathcal{S}_l=\Psi \backslash \{(v_1(s,x),\lambda_1(s),\epsilon_1(s)) \vert s \in (0,\delta)\} $ containing $\{(v_1(s,x),\lambda_1(s),\epsilon_1(s)) \vert s\in (-\delta,0]\}$, then we can readily see that $\mathcal{S}=\mathcal{S}_u \cup \mathcal{S}_l$, $\mathcal{S}_u \cap \mathcal{S}_l=\{(v_1,\lambda_1,\epsilon_1)\}$.  Moreover, we introduce the following four subsets:
\[\mathcal{S}_u^0:(\mathcal{S}_u \backslash \{(\bar{v},\bar{\lambda},\epsilon_1)\}) \cap \{\epsilon >0\} ,~~\mathcal{P}_1^+:\{(v,\lambda) \in \mathcal{X} \times \mathbb{R}^+~\vert~v(x),~v'(x)>0, x\in(0,L)\},\]
\[\mathcal{S}_l^0: (\mathcal{S}_l \backslash \{(\bar{v},\bar{\lambda},\epsilon_1\})\cap \{\epsilon >0\},~~\mathcal{P}_1^-:\{(v,\lambda)\in  \mathcal{X} \times \mathbb{R}^+~\vert~ v(x)>0,~v'(x)<0, x\in(0,L)\},\]
and we want to show that
\[\mathcal{S}_u^0 \subset \mathcal{P}_1^+ \times \mathbb{R}^+,~\mathcal{S}_l^0 \subset \mathcal{P}_1^- \times \mathbb{R}^+. \]
We shall only prove the first part, while the latter one can be treated in the same way.  We first note that $\mathcal{S}_u^0 \neq \emptyset$ since any solution $(v,\lambda,\epsilon)$ of (\ref{1}) near $(\bar{v},\bar{\lambda},\epsilon_1)$ is in the set $\mathcal{S}_u^0 $ thanks to (\ref{14}).  Now that $\mathcal{S}_u^0$ is a connected subset of $ \mathcal{X}^+ \times \mathbb{R}^+   \times \mathbb{R}^+$, we only need to show that $\mathcal{S}_u^0 \cap (\mathcal{P}_1^+ \times \mathbb{R}^+)$ is both open and closed with respect to the topology of $\mathcal{S}_u^0 $ and we divide our proof into two parts.

\textbf{\emph{Openness}}:  Assume that $(\tilde{v}, \tilde{\lambda}, \tilde{\epsilon}) \in \mathcal{S}_u^0 \cap (\mathcal{P}_1^+ \times \mathbb{R}^+)$ and there exists a sequence
$\{(\tilde{v}_k, \tilde{\lambda}_k, \tilde{\epsilon}_k) \}_{k=1}^\infty$  in $\mathcal{S}_u^0$ that converges to $(\tilde{v}, \tilde{\lambda}, \tilde{\epsilon})$ in the norm of $\mathcal{X} \times \mathbb{R}^+ \times \mathbb{R}^+$.  We want to show that for all $k$ large that
$(\tilde{v}_k, \tilde{\lambda}_k, \tilde{\epsilon}_k) \in  \mathcal{P}_1^+ \times \mathbb{R}^+$, i.e,

\[\tilde {v}_k'>0,x\in (0,L),~\tilde{\lambda}_k>0~,\tilde{\epsilon}_k>0.\]
First of all, it is easy to see that $\tilde{\lambda}_k>0~,\tilde{\epsilon}_k>0$ since both have positive limits as $\tilde{\lambda}>0~,\tilde{\epsilon}>0$.  On the other hand, we conclude from $\tilde{v}_k \rightarrow \tilde{v}$ in $\mathcal{X}$ and the elliptic regularity theory that $ \tilde{v}_k \rightarrow  \tilde{v} \text{~in~} C^2([0,L])$.  Differentiate the first equation in (\ref{1}) and we have
\begin{equation}\label{s13}
\left\{
\begin{array}{ll}
-\epsilon \tilde{v}'_{xx}+(2c_2\tilde{v}+\frac{b_2\lambda}{(1+\tilde v)^2}-a_2)\tilde{v}'=0, x \in (0,L),\\
\tilde{v}'(0)=\tilde{v}'(L)=0.
\end{array}
\right.
\end{equation}
We have from Hopf's lemma and the fact $\tilde{v}'(x)>0$ that $\tilde{v}''(L)>0>\tilde{v}''(0)$, then this second order non-degeneracy implies that $\tilde{v}'_k(x)>0$, which is desired.

\textbf{\emph{Closedness}}:  To show that $\mathcal{S}_u^0 \cap (\mathcal{P}_1^+ \times \mathbb{R}^+)$ is closed in $\mathcal{S}_u^0$, we take a sequence $(\tilde{v}_k, \tilde{\lambda}_k, \tilde{\epsilon}_k) \in \mathcal{S}_u^0 \cap (\mathcal{P}_1^+ \times \mathbb{R}^+)$ and assume that there exists $(\tilde{v}, \tilde{\lambda}, \tilde{\epsilon})$ in $\mathcal{S}_u^0$ such that $(\tilde{v}_k, \tilde{\lambda}_k, \tilde{\epsilon}_k)  \rightarrow (\tilde{v}, \tilde{\lambda}, \tilde{\epsilon})$ in the topology of $\mathcal{X} \times \mathbb{R}^+ \times \mathbb{R}^+.$  We want to show that \[\tilde{v}'(x)>0, x\in(0,L),\text{ and }\tilde{\lambda}>0.\]
$\tilde{\lambda}>0$ can be easily proved by the same argument as above and we now need to show that $\tilde{v}'(x)>0$.  Again we have from the elliptic regularity that $\tilde{v}_k  \rightarrow  \tilde{v}$ in $C^2([0,L])$,
therefore $\tilde{v}'(x)\geq 0, x\in(0,L)$.  Applying the Strong Maximum Principle and Hopf's Lemma to (\ref{s13}), we have that either $\tilde{v}'>0$ or $\tilde{v}'\equiv0$ on $(0,L)$.  In the latter case, we must have $(\tilde{v},\tilde{\lambda}) \equiv (\bar{v},\bar{\lambda})$ and this contradicts to the definition of $\mathcal{S}_u^0$.  Thus we have shown that $\tilde v'>0$ on $(0,L)$ and this finishes the proof of (iii).

According to Theorem 4.4 in Shi and Wang \cite{SW}, $\mathcal{S}_{u}$ satisfies one of the following alternatives: $(a1)$  it is not compact in $\mathcal{X}  \times \mathbb{R}^+ \times \mathbb{R}^+$; $(a2)$ it contains a point $(\bar{v},\bar{\lambda},\epsilon_*)$ with $\epsilon_* \neq \epsilon_1$; $(a3)$ it contains a point $(\bar{v}+\hat{v},\bar{\lambda}+\hat{\lambda},\epsilon)$ where $0\neq (\hat{v},\hat{\lambda}) \in \mathcal{Z}$, and $\mathcal{Z}$ is a closed complement of $\mathcal{N}\left(D_{(v,\lambda)}\mathcal{F}(\bar{v},\bar{\lambda},\epsilon_1) \right)=\text{span}\{ (\cos \frac{\pi x}{L},0)\}$.

If $(a2)$ occurs, then $\epsilon_*$ must be one of the bifurcation values $\epsilon_k$, $k\geq2$ and in a neighborhood of $(\bar{v},\bar{\lambda})$, $v(x)$ has the formula
$v(x)=\bar{v}+s \cos\frac{k\pi x}{L}+o(s)$ according to (\ref{14}).  This contradicts to the monotonicity of $v(x)$ and thus $(a2)$ can not happen.

If $(a3)$ occurs, we can choose the complement to be
\[\mathcal{Z}=\Big\{(v,\lambda)\in \mathcal{X}  \times \mathbb{R}^+ ~\Big \vert \int_0^L v(x)\cos \frac{\pi x}{L}dx =0 \Big\},  \]
However, for any $(v,\lambda) \in \mathcal{Z}$, we have from the integration by parts that
\[0=\int_0^L v(x)\cos \frac{\pi x}{L}  dx =-\frac{L}{\pi} \int_0^L v'(x)  \sin \frac{\pi x}{L} dx<0,\]
and this is also a contradiction.  Therefore we have shown that only alternative $(a1)$ occurs and $\mathcal{S}_{u}$ is not compact in $\mathcal{X}  \times \mathbb{R}^+ \times \mathbb{R}^+$.  Now we will study the behavior of $\mathcal{S}_u$ and that of $\mathcal{S}_l$ can be obtained in the exact same way. First of all, we claim that the project of $\mathcal{S}_u$ onto the $\epsilon$-axis does not contain an interval in the form $(\epsilon_0,\infty)$ for any $\epsilon_0>0$ and it is sufficient to show that there exist a positive constant $\bar \epsilon_0$ such that (\ref{1}) has only constant positive solution $(\bar{v},\bar{\lambda})$ if $\epsilon \in (\bar \epsilon_0,\infty)$.  To prove the claim, we decompose solution $v(x)$ of (\ref{1}) as
\[v=\bar{v}_{\text{ave}}+w,\]
where $\bar{v}_{\text{ave}}=\frac{1}{L} \int_0^L vdx$ and $\int_0^L w dx=0$.  Then we readily see that $w$ satisfies
\begin{equation}\label{s14a}
\left\{
\begin{array}{ll}
\epsilon  w''+(a_2-\frac{b_2 \lambda}{1+\bar v+w}-c_2 \bar v-c_2 w )(\bar v+w)=0,~x \in (0,L),\\
\int_0^L w dx =0,~ w'(0)=w'(L)=0.
\end{array}
\right.
\end{equation}
Multiplying both hand sides of (\ref{s14a}) by $w$ and then integrating over $(0,L)$, we have that
\[\epsilon \int_0^L (w')^2 dx=(a_2-2c_2\bar{v})\int_0^L w^2 dx-b_2\lambda \int_0^L \frac{(\bar v+w)w}{1+\bar v+w}dx-c_2\int_0^L w^3dx.\]
We can easily show from the Maximum Principle that both $v(x)$ and $\lambda$ are uniformly bounded in $\epsilon$, then we have from the inequality above that
\[\epsilon \int_0^L (w')^2 dx \leq C_0(a_2,c_2,\bar{v})\int_0^L w^2 dx.\]
Then we reach a contradiction for all $\epsilon>\frac{C_0}{(\pi/L)^2}$ unless $w\equiv0$, where $(\pi/L)^2$ is the first positive eigenvalue of $-\frac{d^2}{dx^2}$ subject to Neumann boundary condition.  If $w\equiv0$, we have that $v\equiv \bar{v}$ and this is a contradiction as we have shown in the case $(a2)$.  Therefore the claim is proved.

Now we proceed to show that the projection $\mathcal{S}_u$ onto the $\epsilon$-axis is of the form $(0,\bar{\epsilon}]$ for some $\bar{\epsilon} \geq \epsilon_1$.  We argue by contradiction and assume that there exists $\underline \epsilon >0$ such that $(\underline \epsilon ,\bar{\epsilon})$ is contained in this projection, but this projection does not contain any $\epsilon<\underline  \epsilon $.  Then we have from the uniformly boundedness of $\Vert v_\epsilon(x)\Vert_\infty$ and sobolev embedding that, for each $\epsilon>0$, $\Vert v_\epsilon \Vert_{C^3([0,L])} \leq C$, $\forall(v_\epsilon,\lambda_\epsilon,\epsilon) \in \mathcal{S}_u$ and this implies that $\mathcal{S}_u$ is compact in $\mathcal{X} \times \mathbb{R}^+ \times \mathbb{R}^+$.  We reach a contradiction to alternative $(a1)$.  Therefore $\mathcal{S}_u$ extends to infinity vertically in $\mathcal{X}  \times \mathbb{R}^+ \times \mathbb{R}^+$.  This finishes the proof of (iii) and Theorem \ref{thm31}.
\end{proof}

We have from Theorem \ref{thm31} that there exist positive and monotone solutions $v_\epsilon(\lambda_\epsilon,x)$ to (\ref{1}) for all $\epsilon\in(0,\epsilon_1)$.  If $v(x)$ is a monotone increasing solution to (\ref{1}), $v(L-x)$ is a decreasing solution.  Then we can construct infinitely many non-monotone- solutions of (\ref{1}) by reflecting and periodically extending $v(x)$ and $v(L-x)$ at $...,-L,0,L,...$

%

\section{Stability of bifurcating solutions from $(\bar v,\bar \lambda, \epsilon_k)$}

In this section, we proceed to investigate the stability or instability of the spatially inhomogeneous solution $(v_k(s,x),\lambda_k(s))$ that bifurcates from $(\bar v,\bar \lambda)$ at $\epsilon=\epsilon_k$.  Here the stability refers to that of the inhomogeneous pattern taken as an equilibrium to the time-dependent system of (\ref{1}).   To this end, we apply the classical results from Crandall and Rabinowitz \cite{CR2} on the linearized stability with an analysis of the spectrum of system (\ref{1}).

First of all, we determine the direction to which the bifurcation curve $\Gamma_k(s)$ turns to around $(\bar v,\bar \lambda,\epsilon_k)$.  According to Theorem 1.7 from \cite{CR}, the bifurcating solutions $(v_k(s,x),\lambda_k(s),\epsilon_k(s))$ are smooth functions of $s$ and they can be written into the following expansions
\begin{equation}\label{18}
  \left \{
\begin{array}{ll}
v_k(s,x)=\bar v + s \cos\frac{k\pi x}{L}+s^2\varphi_2+s^3\varphi_3 +o(s^3),\\
\lambda_k(s,x)=\bar \lambda + s^2 \bar \lambda_2 +s^3 \bar \lambda_3+o(s^3),\\
\epsilon_k(s)= \chi_k+ s \mathcal{K}_1+s^2\mathcal{K}_2+o(s^2),\\
\end{array}
\right.
\end{equation}
where $\varphi_i \in H^2(0,L)$ satisfies that $\int_0^L \varphi_i \cos \frac{k\pi x}{L}dx=0$ for $i=2,3$,  and $\bar \lambda_2$, $\mathcal{K}_1$, $\mathcal{K}_2$ are positive constants to be determined.  We remind that $o(s^3)$ in the $v$--equation of (\ref{18}) is taken in the norm of $H^2(0,L)$.  For notational simplicity, we denote in (\ref{1})
\begin{equation}\label{19}
f(v,\lambda)=\big(a_2-\frac{b_2 \lambda}{1+v}-c_2v\big)v,~g(v,\lambda)=\frac{a_1-c_1v}{1+v}-\frac{b_1 \lambda}{(1+v)^2}.
\end{equation}
Moreover, we introduce the notations
\begin{equation}\label{20}
\bar f_v=\frac{\partial f(v,\lambda)}{\partial v}\vert_{(v,\lambda)=(\bar v,\bar\lambda)},~~\bar f_\lambda=\frac{\partial f(v,\lambda)}{\partial \lambda}\vert_{(v,\lambda)=(\bar v,\bar\lambda)}
\end{equation}
and we can define $\bar f_{v\lambda}, \bar f_{vvv}, \bar g_{v}, \bar g_{\lambda}, \bar g_{v\lambda}, \bar g_{vv}$, etc. in the same manner.  Our analysis and calculations are heavily involved with these values and we also want to remind our reader that $f(\bar v,\bar \lambda)=g(\bar v,\bar \lambda)=0$.

By substituting (\ref{18}) into (\ref{1}) and collecting the $s^2$-terms, we obtain that
\begin{equation}\label{21}
\epsilon_k \varphi''_2-\mathcal{K}_1 \Big(\frac{k\pi}{L}\Big)^2\cos\frac{k\pi x}{L} +\bar f_v \varphi_2 +\bar f_\lambda \bar  \lambda_2+\frac{1}{2}\bar f_{vv}\cos^2\frac{k\pi x}{L}=0.
\end{equation}
Multiplying (\ref{21}) by $\cos\frac{k\pi x}{L}$ and integrating it over $(0,L)$ by parts, we see that
\[\frac{k^2\pi^2}{2L}\mathcal{K}_1=\Big(-\epsilon_k\Big(\frac{k\pi}{L}\Big)^2+\bar f_v\Big) \int_0^L \varphi_2 \cos \frac{k\pi x}{L}dx =0,\]
therefore $\mathcal{K}_1=0$ and the bifurcation at $(\bar v,\bar \lambda, \epsilon_k)$ is of pitch-fork type for all $\epsilon_k$, $k\in N^+$.

By collecting the $s^3$-terms from (\ref{1}), we have
\begin{equation}\label{22}
\begin{array}{ll}
&\epsilon_k \varphi''_3+\bar f_v \varphi_3+\bar f_\lambda \bar \lambda_3-\mathcal{K}_2 \Big(\frac{k\pi}{L}\Big)^2\cos\frac{k\pi x}{L}\\
+&\Big(\bar f_{vv}\varphi_2+\bar f_{v\lambda}\bar\lambda_2  \Big)\cos\frac{k\pi x}{L}+\frac{1}{6} \bar f_{vvv}\cos^3\frac{k\pi x}{L}=0.
\end{array}
\end{equation}

Testing (\ref{22}) by $\cos\frac{k\pi x}{L}$, we conclude through straightforward calculations that
\begin{equation}\label{23}
\frac{k^2\pi^2}{2L}\mathcal{K}_2=\frac{1}{2}\bar f_{vv}   \Big(\int_0^L \varphi_2 \cos \frac{2k\pi x}{L}dx+\int_0^L \varphi_2 dx\Big)+\frac{L}{2} \bar f_{v\lambda} \bar \lambda_2+\frac{L}{16}\bar f_{vvv}.
\end{equation}
Therefore, we need to evaluate the integrals $\int_0^L \varphi_2 \cos \frac{2k\pi x}{L}dx $ and $\int_0^L \varphi_2 dx $ as well as $\bar \lambda$ to find the value of $\mathcal{K}_2$.

Multiplying (\ref{21}) by $\cos\frac{2k\pi x}{L}$ and then integrating it over $(0,L)$ by parts, since $\mathcal{K}_1=0$, we have from straightforward calculations that
\begin{equation}\label{24}
\int_0^L \varphi_2 \cos \frac{2k\pi x}{L}dx=\frac{L}{24\epsilon_k}(\frac{L}{k\pi})^2 \bar f_{vv}=\frac{L^3}{12\epsilon_k (k\pi)^2}\Big(\frac{b_2\bar \lambda}{(1+\bar v)^3}-c_2\Big).
\end{equation}
Integrating (\ref{21}) over $(0,L)$ by parts, we have that
\begin{equation}\label{25}
\bar f_v \int_0^L \varphi_2 dx+\bar f_\lambda \bar \lambda_2 L+\frac{L}{4}\bar f_{vv}=0,
\end{equation}
where we have applied in (\ref{25}) the fact that $\bar f_v=\epsilon_k \big(\frac{k\pi}{L}\big)^2$ in order to keep the solution in a neat form in the coming calculations.  Furthermore, we collect $s^2$ terms from the integration equation in (\ref{1}) and have that
\begin{equation}\label{26}
\bar g_v \int_0^L \varphi_2 dx+\bar g_\lambda \bar \lambda_2 L+\frac{L}{4}\bar g_{vv}=0.
\end{equation}
Solving (\ref{25}) and (\ref{26}) leads us to
\begin{equation}\label{27}
\int_0^L \varphi_2 dx =\frac{(\bar f_\lambda\bar g_{vv}-\bar f_{vv} \bar g_\lambda)L }{4(\bar f_v\bar g_\lambda -\bar f_\lambda\bar g_v)},~\bar \lambda_2 =\frac{ \bar f_{vv} \bar g_{v}-\bar f_{v} \bar g_{vv}  }{4(\bar f_v\bar g_\lambda -\bar f_\lambda\bar g_v)}.
\end{equation}

By substituting (\ref{24}) and (\ref{27}) into (\ref{23}), we obtain that
\begin{equation}\label{28}
\frac{(k\pi)^2}{2L^2} \mathcal{K}_2=\frac{\bar f_{vv} \bar g_{vv}\bar f_\lambda-\bar f_{v\lambda}(\bar f_v\bar g_{vv}-\bar g_v\bar f_{vv})}{8(\bar f_v \bar g_\lambda -\bar g_v\bar f_\lambda)}+\frac{\bar f^2_{vv}}{48 \bar f_v}+\frac{\bar f_{vvv}}{16}.
\end{equation}
On the other hand, we have from straightforward calculations that
\begin{equation}\label{29}
\bar f_v=\frac{b_2\bar \lambda \bar v}{(1+\bar v)^2}-c_2\bar v,~\bar f_\lambda=-\frac{b_2 \bar v}{1+\bar v},~\bar f_{v\lambda}=-\frac{b_2}{(1+\bar v)^2},
\end{equation}
\begin{equation}\label{30}
\bar f_{vv}=\frac{2b_2\bar \lambda}{(1+\bar v)^3}-2c_2,~\bar f_{vvv}=-\frac{6b_1 \bar \lambda}{(1+\bar v)^4},
\end{equation}
and
\begin{equation}\label{31}
\bar g_v=-\frac{a_1+c_1}{(1+\bar v)^2}+\frac{2b_1 \bar \lambda}{(1+\bar v)^3},~\bar g_\lambda=-\frac{b_1  }{(1+\bar v)^2},~\bar g_{vv}=\frac{2(a_1+c_1)}{(1+\bar v)^3}-\frac{6b_1 \bar \lambda}{(1+\bar v)^4}.
\end{equation}
Moreover, we can also have that
\begin{equation}\label{32}
\bar f_v\bar g_\lambda -\bar f_\lambda\bar g_v=\frac{(b_1c_2-b_2c_1)\bar v}{(1+\bar v)^2}.
\end{equation}
For the simplicity of calculations, we assume that $b_1=0$ from now on.  Substituting (\ref{29})--(\ref{31}) into (\ref{32}), we have that
\begin{eqnarray}\label{33}
\frac{(k\pi)^2}{2L^2} \mathcal{K}_2 \!\!\!\!\!\!   &= \!\!\!\!\!\!   & \frac{\frac{2(a_1+c_1)}{(1+\bar v)^3}\big(\frac{2b_2\lambda}{(1+\bar v)^3}\!-\!2c_2 \big)\big(\frac{-b_2\bar v}{1+\bar v} \big)\!+\!\frac{b_2 }{(1+\bar v)^2}\big(\frac{b_2 \bar \lambda}{1+\bar v}\!-\!c_2(1+2\bar v) \big)  }{-8b_2c_1\bar v/(1+\bar v)^2 }\!+\!\frac{\bar f^2_{vv}}{48\bar f_v}\!+\!\frac{\bar f_{vvv}}{16}     \nonumber \\
   \!\!\!\!\!\! &=\!\!\!\!\!\!  &  -\frac{(a_1+c_1)\frac{b_2\bar \lambda(1-2\bar v) }{(1+\bar v)^3}+\frac{2c_2\bar v^2-c_2}{1+\bar v} }{4c_1\bar v(1+\bar v)^2}+\frac{\big(\frac{2b_2\bar \lambda}{(1+\bar v)^3}-2c_2\big)^2}{48(\frac{b_2\bar \lambda \bar v}{(1+\bar v)^2}-c_2\bar v)}-\frac{3 b_1 \bar \lambda}{8(1+\bar v)^4}.
\end{eqnarray}

For the simplicity of notations, we introduce
\[t=\frac{b_2\bar \lambda}{(1+\bar v)^2}=\frac{a_2-c_2(1+\bar v)}{1+\bar v},\]
then one can easily see that (\ref{13}) implies that bifurcation occurs at $(\bar v,\bar \lambda,\epsilon_k)$ if and only if $t>c_2$ and we shall assume that $t>c_2$ from now on.  In terms of the new variable $t$, we observe that (\ref{33}) becomes
\begin{eqnarray}\label{34}
\frac{(k\pi)^2}{2L^2} \mathcal{K}_2 \!\!\!&=\!\!\!& -\frac{\frac{ 1-2\bar v }{ 1+\bar v }t+\frac{2c_2\bar v^2-c_2}{1+\bar v} }{4c_1\bar v(1+\bar v) }+\frac{\big(\frac{t}{1+\bar v} -c_2\big)^2}{12\bar v (t-c_2)}-\frac{3t}{8(1+\bar v)^2}     \nonumber \\
   \!\!\!&=\!\!\!&  \frac{-\frac{3(t-c_2)}{1+\bar v} \Big(\frac{ 1-2\bar v }{ 1+\bar v }t+\frac{2c_2\bar v^2-c_2}{1+\bar v}   \Big)+\big(\frac{t}{1+\bar v} -c_2\big)^2-\frac{9\bar v t(t-c_2)}{2(1+\bar v)^2}  }{12\bar v (t-c_2)}  \\
   \!\!\!&=\!\!\!&\frac{F(t)}{24 \bar v(1+\bar v)^2 (t-c_2)}=\frac{\alpha t^2+\beta t+\gamma}{24 \bar v(1+\bar v)^2 (t-c_2)},~t>c_2, \nonumber
\end{eqnarray}
where in the last line of (\ref{34}) we have used the notations
\begin{equation}\label{35}
\alpha=3\bar v-4,~\beta=-(12\bar v^2+7\bar v-8)c_2,~\gamma=(14\bar v^2+4\bar v-4)c^2_2.
\end{equation}
Now we are ready to determine the sign of $\mathcal{K}_2$ which is crucial in the stability analysis of $(v_k(x,s),\lambda_k(s),\epsilon_k)$ as we shall see later.  To this end, we first have from straightforward calculations that
\[F(c_2)=2 c_2^2\bar v^2>0;\]
moreover, if $\bar v\neq \frac{4}{3}$, the quadratic function $F(t)$ has its determinant $(144\bar v^4+33\bar v^2)c_2^2$ and $F(t)=0$ always have two roots which are
\begin{eqnarray}\label{36}
t_1^*=\frac{\big((12 \bar v^2+7\bar v-8)-\sqrt{144\bar v^4+33\bar v^2}\big)c_2}{2(3\bar v -4)},\nonumber  \\
t_2^*=\frac{\big((12 \bar v^2+7\bar v-8)+\sqrt{144\bar v^4+33\bar v^2}\big)c_2}{2(3\bar v -4)};
\end{eqnarray}
furthermore, we readily see that $-\frac{\beta}{2\alpha}-c_2=\frac{(12\bar v^2+\bar v)c_2}{2\alpha}$ and it implies that $t_1^*<c_2<t^*_2$ if $\bar v \in (0,\frac{4}{3})$ and $c_2<t_1^*<t^*_2$ if $\bar v \in (\frac{4}{3},\infty)$.  In particular, if $\bar v =\frac{4}{3}$, we have that $F(t)=\beta t+\gamma=-\frac{68c_2}{3}t+\frac{266c_2^2}{9}$ and it has a unique positive root $\frac{59c_2}{51}$.  Then we have the following results on the signs of $\mathcal{K}_1$ and $\mathcal{K}_2$.
\begin{proposition}
Suppose that (\ref{13}) holds and the bifurcation solutions takes the form (\ref{18}).  Then $\mathcal{K}_1=0$ and the bifurcation branch is of pitchfork type at $(\bar v,\bar \lambda,\epsilon_k)$ for each $k\in N^+$.  Moreover, we assume that $b_1=0$ and denote $t=\frac{a_2-c_2(1+\bar v)}{1+\bar v}$, then we have that, \\
if $\bar v \in (0,\frac{4}{3})$, $\mathcal{K}_2>0$ for $t\in(c_2,t^*_2) $ and $\mathcal{K}_2<0$ for $t\in(t^*_2,\infty)$; \\
if $\bar v =\frac{4}{3}$, $\mathcal{K}_2>0$ for $t\in(c_2,\frac{59 c_2}{51}) $ and $\mathcal{K}_2<0$ for $t\in(\frac{59 c_2}{51},\infty)$;\\
if $\bar v \in ( \frac{4}{3},\infty)$, $\mathcal{K}_2>0$ for $t\in(c_2,t^*_1) \cup (t^*_2,\infty)$ and $\mathcal{K}_2<0$ for $t\in(t^*_1,t^*_2)$.
\end{proposition}
The graphes of $\mathcal{K}_2$ as a function $t$ are illustrated in Figure (1).  It should be observed that, $\mathcal{K}_2>0$ for $t$ slightly bigger than $c_2$ since the bifurcation value $\epsilon_k\ll 1$ in this situation and we have that $\mathcal{K}_2$ is always positive regardless of $\bar v$.

 \begin{remark}
We want to note that the assumption $b_1=0$ is made only for the sake of mathematical simplicity while $\mathcal{K}_2$ becomes extremely complicated to calculate if $b_1>0$ in (\ref{1}).  On the other hand, we will shall see in Section 4 that, even when $b_1=0$, system (\ref{1}) admits solutions with single transition layer for $\epsilon$ being sufficiently small.  Moreover, this limiting condition is also necessary in our analysis of the transition-layer solutions in Section 4.
\end{remark}

\begin{figure}[!htb]
\minipage{.4\textwidth}
  \includegraphics[width=1.5in]{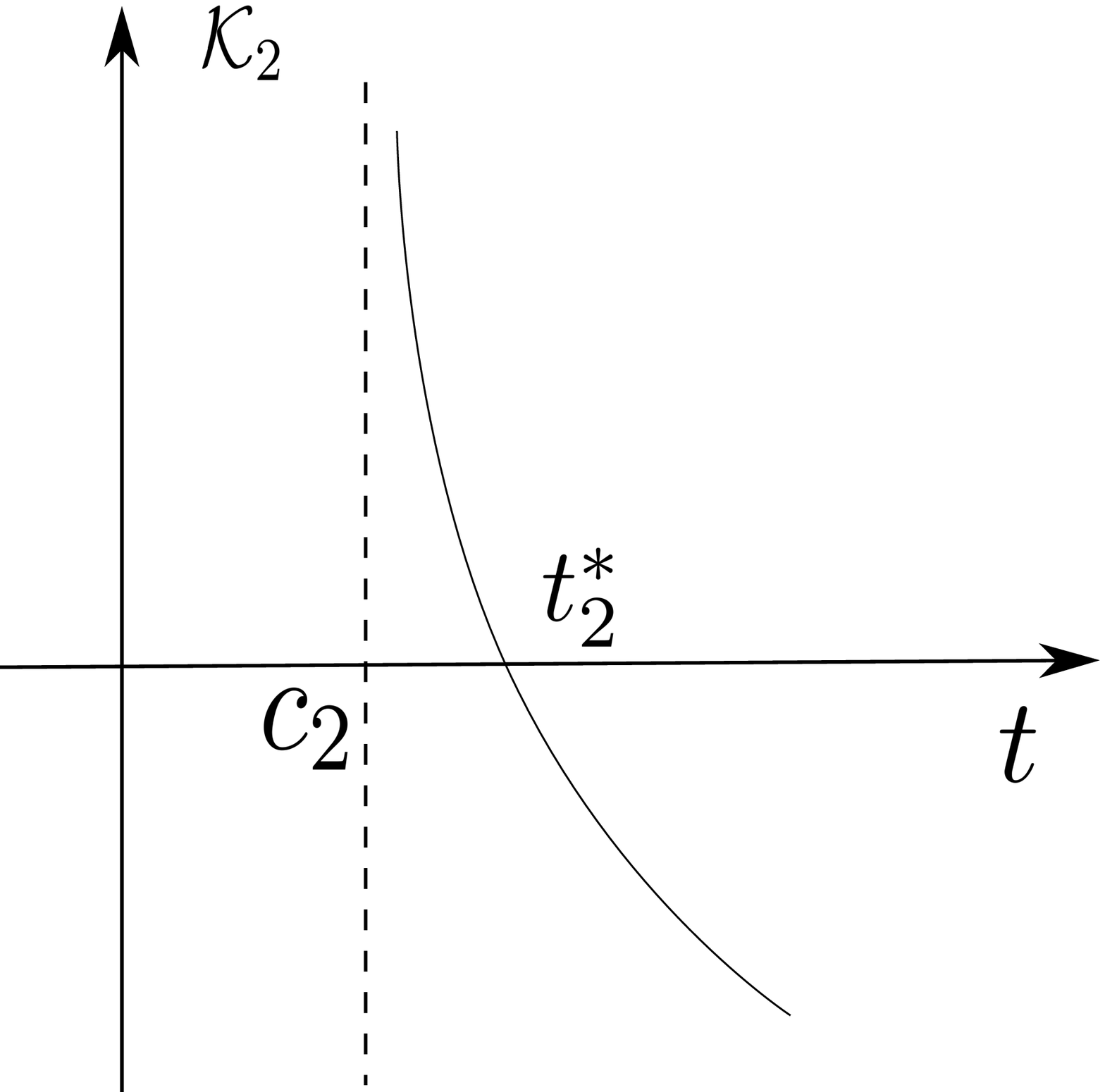}
  \caption*{$\bar v\in(0,\frac{4}{3})$}\label{fig:awesome_image1}
\endminipage
\minipage{.4\textwidth}
  \includegraphics[width=1.5in]{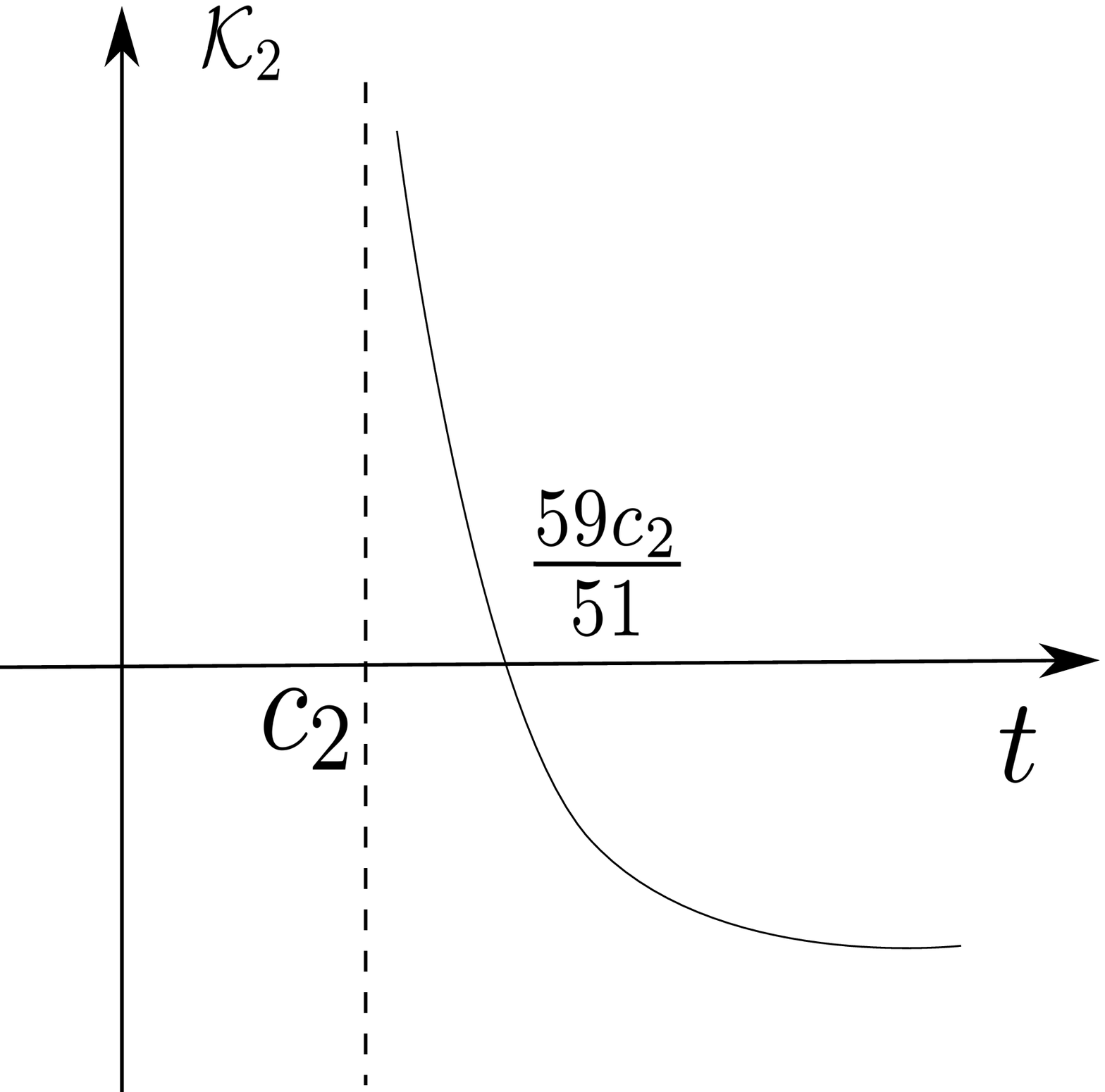}
  \caption*{$\bar v=\frac{4}{3}$}\label{fig:awesome_image2}
\endminipage
\minipage{.4\textwidth}%
  \includegraphics[width=1.5in]{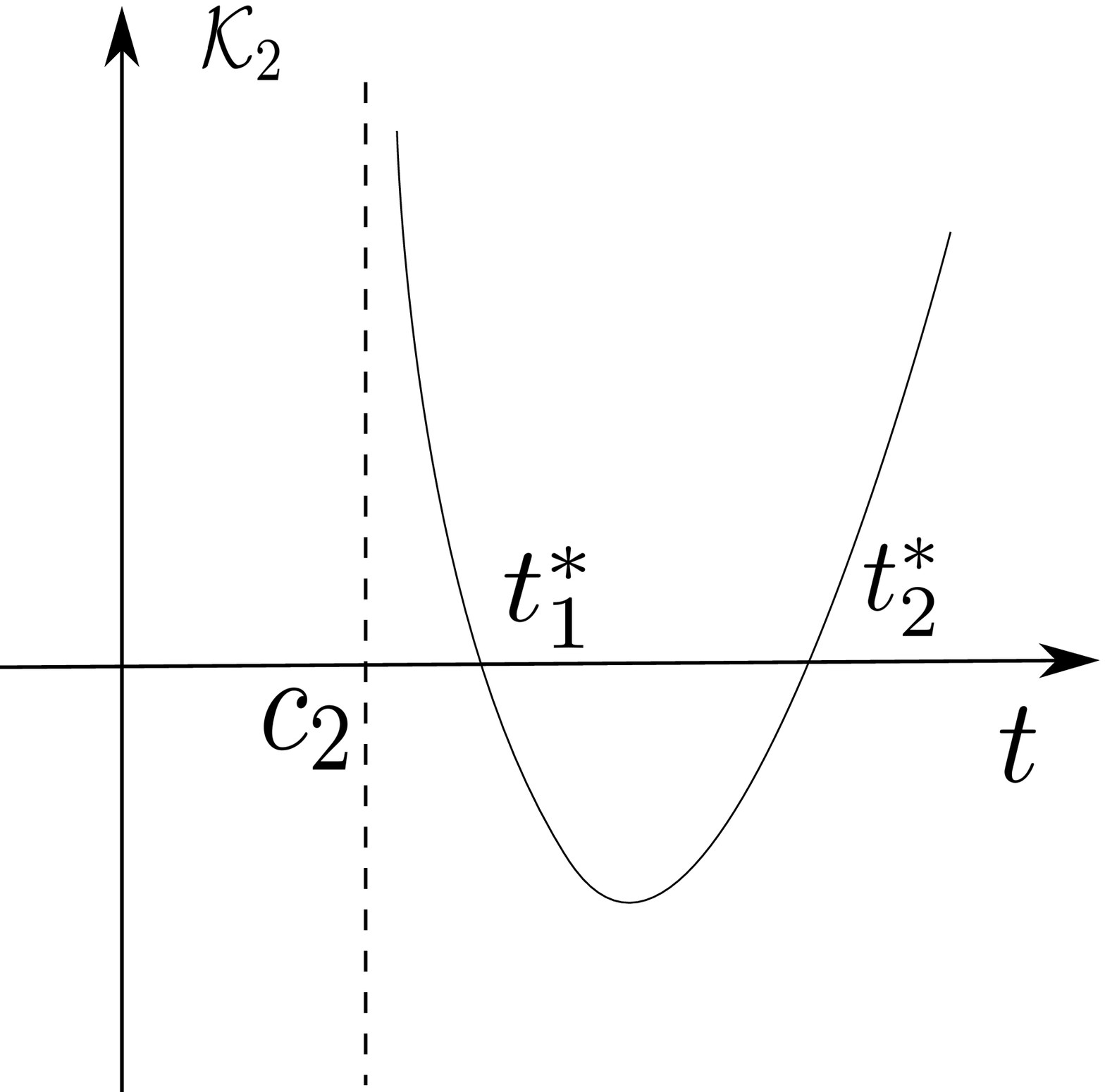}
  \caption*{$\bar v\in(\frac{4}{3},\infty)$}\label{fig:awesome_image3}
\endminipage
\caption{Graphs of $\mathcal{K}_2$ as a function of $t=\frac{a_2-c_2\bar v}{1+\bar v}$.  The assumption that $t>c_2$ is required since bifurcation occurs at $(\bar v,\bar \lambda,\epsilon_k)$ only when (\ref{13}) holds.  }
\end{figure}

We are ready to present the following result on the stability of the bifurcation solution $(v_k(s,x),\lambda_k(s))$ established in Theorem 2.2.  Here the stability refers to the stability of the inhomogeneous solutions taken as an equilibrium to the time-dependent counterpart to (\ref{1}).
\begin{theorem}\label{thm31}
Assume that (\ref{13}) is satisfied.  Then for each $k\in N^+$, if $\mathcal{K}_2>0$, the bifurcation curve $\Gamma_k(s)$ at $(\bar v,\bar \lambda,\epsilon_k)$ turns to the right and the bifurcation solution $(v_k(s,x),\lambda_k(s))$ is unstable and if $\mathcal{K}_2<0$, $\Gamma_k(s)$ turns to the left and $(v_k(s,x),\lambda_k(s))$ is asymptotically stable.
\end{theorem}

The bifurcation branches described in Theorem \ref{thm31} are formally presented in Figure 2.  The solid curve means stable bifurcation solutions and the dashed means the unstable solution.

To study the stability of the bifurcation solution from $(\bar v,\bar \lambda,\epsilon_k)$, we linearize (\ref{1}) at $(v_k(s,x),\lambda_k(s),\epsilon_k(s))$.  By the principle of the linearized stability in Theorem 8.6 \cite{CR}, to show that they are asymptotically stable, we need to prove that the each eigenvalue $\eta$ of the following elliptic problem has negative real part:
\[D_{(v,\lambda)}\mathcal{F}(v_k(s,x),\lambda_k(s),\epsilon_k(s))(v,\lambda)=\eta (v,\lambda),~(v,\lambda)\in\mathcal{X} \times \mathbb{R}.\]
We readily see that this eigenvalue problem is equivalent to
\begin{equation}\label{37}
\left\{
\begin{array}{ll}
\epsilon_k(s)  v''+\big(a_2-\frac{b_2\lambda_k(s)}{(1+v_k(s,x))^2}-2c_2v_k(s,x) \big)v- \frac{b_2v_k(s,x)}{1+v_k(s,x)}\lambda=\eta v,&x\in(0,L), \\
\int_0^L \big(\frac{2b_1\lambda_k(s)}{(1+v_k(s,x))^3}-\frac{a_1+c_1}{(1+v_k(s,x))^2} \big)v- \frac{b_1 \lambda}{(1+v_k(s,x))^2} dx=\eta \lambda,\\
v'(0)=v'(L)=0,
\end{array}
\right.
\end{equation}
where $v_k(s,x)$, $\lambda_k(s)$ and $\epsilon_k(s)$ are as established in Theorem 2.2.  On the other hand, we observe that $0$ is a simple eigenvalue of $D_{(v,\lambda)}\mathcal{F}(\bar v,\bar \lambda,\epsilon_k)$ with an eigenspace $\text{span}\{(\cos \frac{k\pi x}{L},0)\}$. It follows from Corollary 1.13 in \cite{CR} that, there exists an internal $I$ with $\epsilon_k\in I$ and continuously differentiable functions $\epsilon\in I\rightarrow \mu(\epsilon),~s\in(-\delta,\delta) \rightarrow \eta(s)$ with $\eta(0)=0$ and $\mu(\epsilon_k)=0$ such that, $\eta(s)$ is an eigenvalue of (\ref{63}) and $\mu(\chi)$ is an eigenvalue of the following eigenvalue problem
\begin{equation}\label{38}
D_{(v,\lambda)}\mathcal{F}(\bar v,\bar \lambda,\epsilon)(v,\lambda)=\mu(v,\lambda),~(v,\lambda)\in \mathcal{X} \times \mathbb{R};
\end{equation}
moreover, $\eta(s)$ is the only eigenvalue of (\ref{37}) in any fixed neighbourhood of the origin of the complex plane (the same assertion can be made on $\mu(\epsilon)$).  We also know from \cite{CR} that the eigenfunctions of (\ref{38}) can be represented by $(v(\epsilon,x),\lambda(\epsilon))$ which depend on $\epsilon$ smoothly and are uniquely determined through $\big(v(\epsilon_k,x),\lambda(\epsilon_k)\big)=\big(\cos \frac{k\pi x}{L},0 \big)$, together with $\big(v(\epsilon,x)- \cos \frac{k\pi x}{L}, \lambda(\epsilon,x)  \big)  \in \mathcal{Z}$.

\begin{proof} [Proof\nopunct] \emph{of Theorem} \ref{thm31}.
Differentiating (\ref{38}) with respect to $\epsilon$ and setting $\epsilon=\epsilon_k$, we arrive at the following system since $\mu(\epsilon_k)=0$
\begin{equation}\label{39}
\left\{
\begin{array}{ll}
-\big(\frac{k\pi}{L}\big)^2\cos \frac{k\pi x}{L}+\epsilon_k \dot{v}''+  \big(  \frac{a_2-c_2-2c_2 \bar{v}}{1+\bar{v}} \big)\bar{v}\dot{v}- \frac{b_2\bar{v}}{1+\bar{v}}\dot{\lambda}=\dot{\mu}(\epsilon_k)\cos \frac{k\pi x}{L},&x\in(0,L), \\
\int_0^L\big( \frac{a_1-c_1-2c_1 \bar{v}}{(1+\bar{v})^2}  \big)\dot{v}- \frac{b_1 \dot{\lambda}}{(1+\bar{v})^2} dx=0,\\
\dot{v}'(0)=\dot{v}'(L)=0,
\end{array}
\right.
\end{equation}
where the dot-sign means the differentiation with respect to $\epsilon$ evaluated at $\epsilon=\epsilon_k$ and in particular $\dot{v}=\frac{\partial v(\epsilon,x)}{\partial \epsilon}\big\vert_{\epsilon=\epsilon_k}$, $\dot{\lambda}=\frac{\partial \lambda(\epsilon,x)}{\partial \epsilon}\big\vert_{\epsilon=\epsilon_k}$.

Multiplying the first equation of (\ref{39}) by $\cos\frac{k\pi x}{L}$ and integrating it over $(0,L)$ by parts, we obtain that
\[\dot{\mu}(\epsilon_k)=-\big(\frac{k\pi}{L}\big)^2.\]
According to Theorem 1.16 in \cite{CR}, the functions $\eta(s)$ and $-s\epsilon'_k(s)\dot{\mu}(\epsilon_k)$ have the same zeros and the same signs for $s\in(-\delta,\delta)$.  Moreover
\[\lim_{s\rightarrow 0,~\eta(s)\neq0}\frac{-s\epsilon'_k(s)\dot{\mu}(\epsilon_k)}{\eta(s)}=1.\]
Now, since $\mathcal{K}_1=0$, it follows that $\lim_{s\rightarrow 0} \frac{s^2\mathcal{K}_2 }{\eta(s)}=\big( \frac{L}{k\pi}\big)^2$ and we readily see that $\text{sgn}(\eta(s))=\text{sgn}(\mathcal{K}_2)$ for $s\in(-\delta,\delta)$, $s\neq0$.  Therefore, we have proved Theorem \ref{thm31} according to the discussion above.

\begin{figure}[!htb]
\minipage{.6\textwidth}
  \includegraphics[width=1.85in]{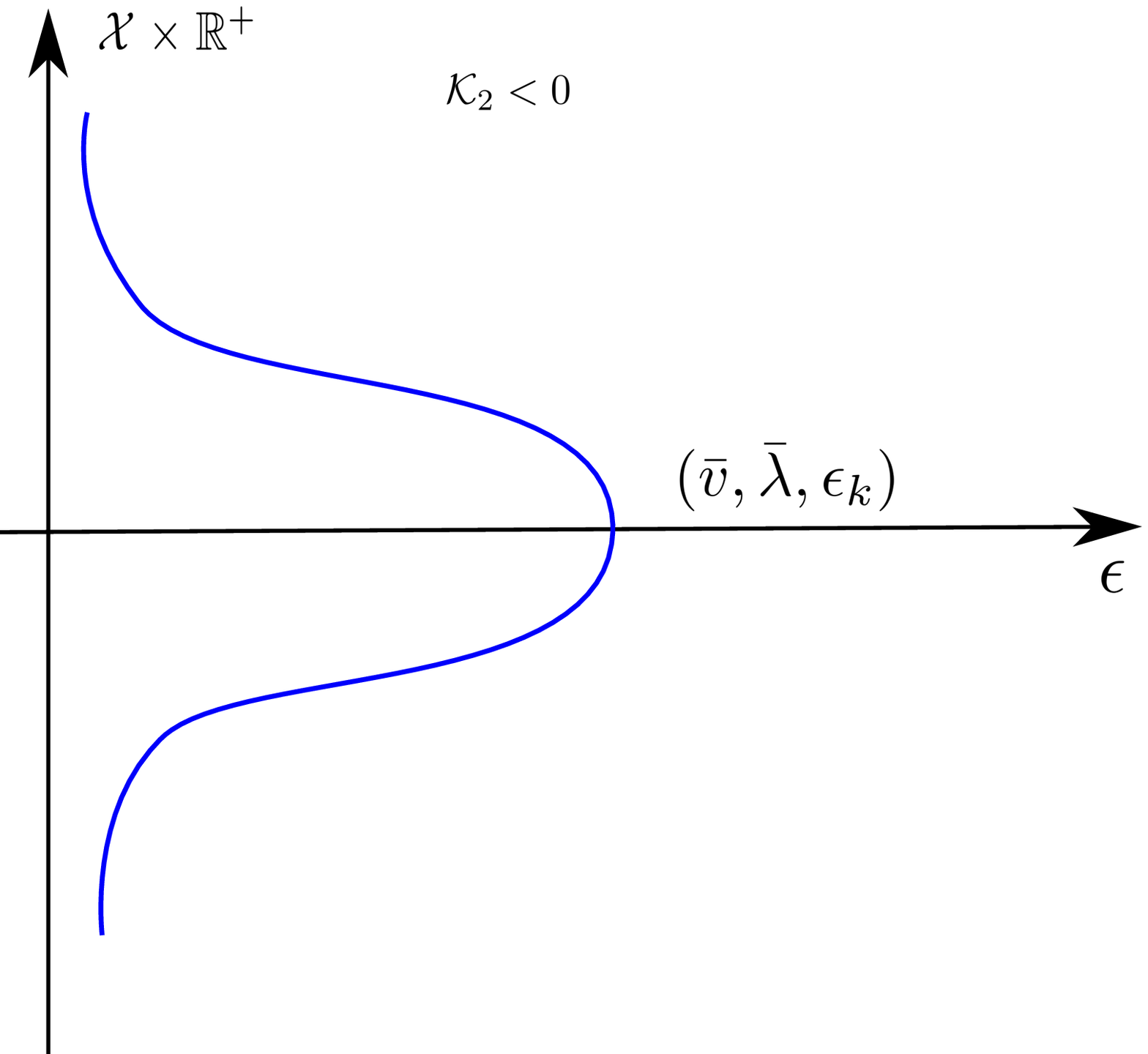}
  \caption*{An illustration of the bifurcation branch for $\mathcal{K}_2<0$}\label{fig:awesome_image1}
\endminipage
\minipage{.6\textwidth}
  \includegraphics[width=1.8in]{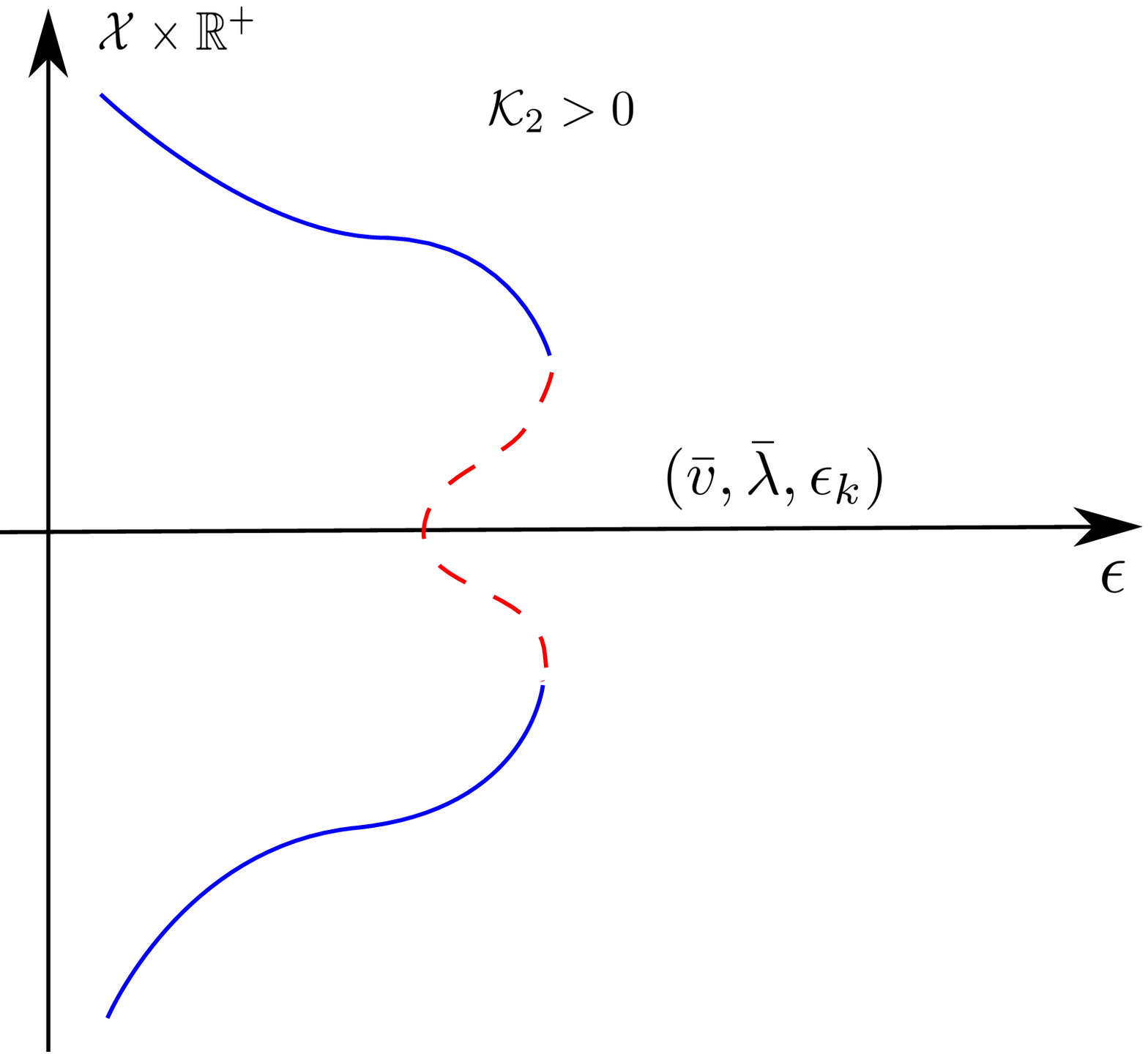}
  \caption*{An illustration of the bifurcation branch for $\mathcal{K}_2>0$}\label{fig:awesome_image2}
\endminipage
\caption{Pitchfork-type bifurcation branches.  The solid line represents stable bifurcating solution $(v_k(s,x),\lambda_k(s),\epsilon_k(s))$ and the dashed line represents unstable solution.}
\end{figure}

\end{proof}

Thanks to (\ref{12}), there always exist nonconstant positive solutions to (\ref{1}) for each $\epsilon$ being small  However, according to Proposition 1 and Theorem \ref{thm31}, the small-amplitude bifurcating solutions $(v_k(s,x),\lambda_k(s),\epsilon_k(s))$ are unstable for $\epsilon_k$ being sufficiently small.  Therefore, we are motivated to find positive solutions to (\ref{1}) that have large amplitude.


\section{Existence of transition-layer solutions}
In this section, we show that, for $\epsilon$ being sufficiently small, system (\ref{1}) always admits solutions with a single transition layer, which is an approximation of a step-function over $[0,L]$.  For the simplicity of calculations, we assume that $b_1=0$ and consider the following system throughout the section.
\begin{equation}\label{40}
\left\{
\begin{array}{ll}
\epsilon v''+(a_2-\frac{b_2 \lambda}{1+v}-c_2v)v=0,&x \in (0,L),     \\
v'(0)=v'(L)=0,\\
\int_0^L \frac{a_1-c_1v}{1+v}dx= 0.
\end{array}
\right.
\end{equation}
Our first approach is to construct the transition-layer solution $v_{\epsilon}(\lambda,x)$ of (\ref{40}) without the integral constraint, with $\lambda$ being fixed and $\epsilon$ being sufficiently small.   We then proceed to find $\lambda=\lambda_\epsilon$ and $v_\epsilon(\lambda_\epsilon,x)$ such that the integral condition is satisfied.  In particular, we are concerned with $v_\epsilon(x)$ that has a single transition layer over $(0,L)$, and we can construct solutions with multiple layers by reflection and periodic extensions of $v_\epsilon(x)$ at $x=...,-2L,-L,0,L,2L,...$

To this end, we first study the following equation
\begin{equation}\label{41}
\left\{
\begin{array}{ll}
\epsilon v''+f(\lambda, v)=0,&x \in (0,L),     \\
v(x)>0,& x\in(0,L),\\
v'(0)=v'(L)=0,
\end{array}
\right.
\end{equation}
where \[f(\lambda,v)=\Big(a_2-\frac{b_2 \lambda}{1+v}-c_2v\Big)v,\]
and $\lambda$ is a positive constant independent of $\epsilon$.

It is easy to see that (\ref{41}) has three constant solutions $\bar v_0(\lambda)=0, \bar v_1(\lambda)\leq \bar v_2(\lambda)$ for each $\lambda \leq \frac{(a_2+c_2)^2}{4b_2c_2}$, where
\begin{equation}\label{42}
\bar v_1(\lambda)=\frac{a_2\!\!-\!\!c_2\!-\!\sqrt{(a_2\!\!+\!\!c_2)^2\!\!-\!\!4b_2c_2\lambda}}{2c_2},~\bar v_2(\lambda)=\frac{a_2\!\!-\!\!c_2\!+\!\sqrt{(a_2\!\!+\!\!c_2)^2\!\!-\!\!4b_2c_2\lambda}}{2c_2}.
\end{equation}
and $0<\bar v_1(\lambda)<\bar v_2(\lambda)$ if and only if
\begin{equation}\label{43}
\lambda \in \Big(\frac{a_2}{b_2}, \frac{(a_2+c_2)^2}{4b_2c_2} \Big),
\end{equation}
 hence we shall assume (\ref{43}) for our analysis from now on.  We also want to note that $\bar v_1(\lambda)=0$ and $\bar v_2(\lambda)=\frac{a_2-c_2}{c_2}$ if $\lambda=\frac{a_2}{b_2}$ and $\bar v_1(\lambda)=\bar v_2(\lambda)=\frac{a_2-c_2}{2c_2}$ if $\lambda=\frac{(a_2+c_2)^2}{4b_2c_2}$.  The constant solutions $0$ and $\bar v_2$ are stable and $\bar v_1$ is unstable in the corresponding time-dependent system of (\ref{41}).  Moreover, for each $\lambda \in \Big(\frac{a_2}{b_2}, \frac{(a_2+c_2)^2}{4b_2c_2} \Big)$, we know that $f(\lambda,v)$ is of Allen-Cahn type and $f_v(\lambda,0)<0, f_v(\lambda,\bar v_2(\lambda))<0$.  It is well known from the phase plane analysis that, for example see \cite{F} or \cite{Ka}, the following system has a unique smooth solution $V_0(z)$,
\begin{equation}\label{44}
\left\{
\begin{array}{ll}
V_0''+f(\lambda,V_0)=0,~z \in \mathbb{R},    \\
V_0(z) \in (0,\bar v_2(\lambda)),~z\in \mathbb{R},\\
V_0(-\infty)=0,~V_0(\infty)=\bar v_2(\lambda),~V_0(0)=\bar v_2(\lambda)/2;
\end{array}
\right.
\end{equation}
moreover, there exist some positive constants $C,\kappa$ dependent on $\lambda$ such that
\begin{equation}\label{45}
\Big\vert \frac{dV_0(z)}{dz}  \Big\vert \leq Ce^{-\kappa \vert z \vert }, ~z\in \mathbb{R}.
\end{equation}
Now we construct an approximation of $v_\epsilon(\lambda,x)$ to (\ref{41}) by using the unique solution to (\ref{44}) following \cite{HS}.  For each fixed $x_0\in(0,L)$, we denote
\[L^*=\min\{x_0,L-x_0\}\]
and choose the cut-off functions $\chi_0(y)$ and $\chi_1(y)$ of class $C^\infty([-L,L])$ as
\begin{equation}\label{46}
\chi_0(y)=
\left\{
\begin{array}{ll}
1,&\vert y \vert \leq L^*/4,    \\
0,&\vert y \vert \geq L^*/2,\\
\in (0,1),&y\in[-L,L],
\end{array}
\right.
\end{equation}
with $\chi_1(y)=0$ if $y\in[-L,0]$ and $\chi_1(y)=\bar v_2(\lambda)(1-\chi_0(y))$ if $y\in[0,L]$.  Set
\begin{equation}\label{46a}
V_\epsilon(\lambda,x)=\chi_0(x-x_0)V_0(\lambda,\frac{x-x_0}{\sqrt{\epsilon}})+\chi_1(x-x_0).
\end{equation}
We shall show that, for each $\lambda \in \Big(\frac{a_2}{b_2}, \frac{(a_2+c_2)^2}{4b_2c_2} \Big)$, (\ref{41}) has a solution $v_\epsilon$ in the form
\[v_\epsilon(\lambda,x)=V_\epsilon(\lambda,x)+\sqrt \epsilon \Psi(\lambda,x),\]
where $\Psi$ is a smooth function on $[0,L]$.  Then $\Psi$ satisfies
\begin{equation}\label{47}
\mathcal{L}_\epsilon \Psi+\mathcal{G}_\epsilon+\mathcal{H}_\epsilon=0,
\end{equation}
with
\begin{equation}\label{48}
\mathcal{L}_\epsilon=\epsilon\frac{d^2}{dx^2}+f_v(\lambda,V_\epsilon(\lambda,x)),
\end{equation}

\begin{equation}\label{49}
\mathcal{G}_\epsilon=\epsilon^{-\frac{1}{2}} \Big(\epsilon \frac{d^2 V_\epsilon(\lambda,x)}{dx^2}+f(\lambda,V_\epsilon(\lambda,x)) \Big)
\end{equation}
and
\begin{equation}\label{50}
\begin{array}{ll}
\mathcal{H}_\epsilon&=\epsilon^{-\frac{1}{2}} \Big(f(\lambda,V_\epsilon(\lambda,x)+\sqrt \epsilon\Psi)-f(\lambda,V_\epsilon(\lambda,x))-\sqrt\epsilon \Psi f_v(\lambda,V_\epsilon(\lambda,x))\Big)\\
&=\Big(\frac{b_2\lambda}{(1+V_\epsilon(\lambda,x))^2(1+V_\epsilon(\lambda,x)+\sqrt \epsilon \Psi)}-c_2 \Big)\sqrt{\epsilon}\Psi^2.
\end{array}
\end{equation}
As we can see from above, $\mathcal{G}_\epsilon$ and $\mathcal{H}_\epsilon$ measure the accuracy that $V_\epsilon(\lambda,x)$ approximates the solution $v_\epsilon(\lambda,x)$.  We have the following lemmas on these estimates.

\begin{lemma}\label{lem41}
Suppose that $\lambda \in \Big(\frac{a_2}{b_2}+\delta, \frac{(a_2+c_2)^2}{4b_2c_2}-\delta \Big)$ for $\delta>0$ small.  There exist positive constants $C_1=C_1(\delta)>0$ and $\epsilon=\epsilon_1(\delta)>0$ which is small such that, for all $\epsilon\in(0,\epsilon_1(\delta))$
\[\sup_{x\in(0,L)} \vert\mathcal{G}_\epsilon (x)\vert \leq  C_1.  \]
\end{lemma}
\begin{proof}
By substituting $V_\epsilon(\lambda,x)=\chi_0(x-x_0)V_0(\lambda,\frac{x-x_0}{\sqrt{\epsilon}})+\chi_1(x-x_0)$ into $\mathcal{G}_\epsilon (x)$, we have from (\ref{44}) that
\begin{equation}\label{51}
\begin{array}{ll}
&\epsilon\frac{d^2V_\epsilon(\lambda,x)}{dx^2}+f(\lambda,V_\epsilon(\lambda,x))\\
&=f\big(\lambda,V_\epsilon\big)-\chi_0f\big(\lambda,V_0\big)+\epsilon\chi''_0V_0+2\sqrt{\epsilon}\chi'_0V'_0+\epsilon \chi''_1\\
&=f\big(\lambda,V_\epsilon\big)-\chi_0f\big(\lambda,V_0\big)+O(\sqrt{\epsilon})\\
&=f\big(\lambda,\chi_0(x-x_0)V_0(\lambda,(x-x_0)/\sqrt{\epsilon})+\chi_1(x-x_0)\big)\\
&\hspace{0.2in}-\chi_0f\big(\lambda,V_0(\lambda,(x-x_0)/\sqrt{\epsilon}) \big)+O(\sqrt{\epsilon}).
\end{array}
\end{equation}
We claim that $\vert f\big(\lambda,V_\epsilon\big)-\chi_0f\big(\lambda,V_0\big) \vert=O(\sqrt{\epsilon})$ and we divide our discussions into several cases.  For $\vert x-x_0 \vert \leq L^*/4$, it follows easily from the definitions of $\chi_0$ and $\chi_1$ that $f\big(\lambda,V_\epsilon\big)-\chi_0 f\big(\lambda,V_0\big)=0$ .  For $x-x_0\geq L^*/2$ and $x-x_0\leq -L^*/2$, we have that $\chi_0=0,\chi_1=1$ and $\chi_0=\chi_1=0$ respectively.   Hence $f\big(\lambda,V_\epsilon\big)-\chi_0f\big(\lambda,V_0\big)=0$ in both cases.  For $\vert x-x_0 \vert \in (L^*/4, L^*/2)$, since $V_0$ decays exponentially to $\bar v_2(\lambda)$ at $\infty$ and to $0$ at $-\infty$, we must have that, there exists a positive constant $C$ which is uniform in $\epsilon$ such that
\[\vert f\big(\lambda,V_\epsilon\big)-\chi_0f\big(\lambda,V_0\big) \vert \leq C\sqrt{\epsilon}.\]
This proves our claim and Lemma \ref{lem41} follows from (\ref{51}).
\end{proof}
The following properties of $\mathcal{H}_\epsilon$ also follows from straightforward calculations.
\begin{lemma}\label{lem42}
Suppose that $\lambda \in \Big(\frac{a_2}{b_2}+\delta, \frac{(a_2+c_2)^2}{4b_2c_2}-\delta \Big)$ for $\delta>0$ small.  For any $R>0$, there exists $C_2=C_2(\delta,R)>0$ and $\epsilon_2=\epsilon_2(\delta,R)>0$ small such that, if $\epsilon\in (0,\epsilon_2)$ and $\Vert \Psi_i \Vert_\infty\leq R$, $i=1,2$, we have that
\begin{equation}\label{52}
\Vert\mathcal{H}_\epsilon [\Psi_i]\Vert_\infty \leq C_2\sqrt \epsilon \Vert \Psi_i^2 \Vert_\infty,
\end{equation}
\begin{equation}\label{53}
\Vert\mathcal{H}_\epsilon [\Psi_1]-\mathcal{H}_\epsilon [\Psi_2]\Vert_\infty \leq C_2 \sqrt{\epsilon} \Vert \Psi_1^2-\Psi_2 ^2\Vert_\infty.
\end{equation}
\end{lemma}

\begin{proof} Taking $C_2=b_2\lambda+2c_2 $, we can easily show that (\ref{52}) and (\ref{53}) follows from the definition of $\mathcal{H}_\epsilon$ in (\ref{50}).
\end{proof}
We also need the following properties of the linear operator $\mathcal{L}_\epsilon$ defined in (\ref{48}).

\begin{lemma}\label{lem43}
Suppose that $\lambda \in \Big(\frac{a_2}{b_2}+\delta, \frac{(a_2+c_2)^2}{4b_2c_2}-\delta \Big)$ for $\delta>0$ small.  For any $p\in[1,\infty]$, there exist $C_3=C_3(\delta,p)>0$ and $\epsilon_3=\epsilon_3(\delta,p)>0$ small such that, $\mathcal{L}_\epsilon$ with domain $W^{2,p}(0,L)$ has a bounded inverse $\mathcal{L}^{-1}_\epsilon$ and if $\epsilon \in (0,\epsilon_3(\delta,p))$
\[ \Vert \mathcal{L}^{-1}_\epsilon g \Vert_p \leq C_3 \Vert g \Vert_p ,~\forall g\in L^p(0,L).\]
\end{lemma}

\begin{proof}
To show that $\mathcal{L}_\epsilon$ is invertible, it is sufficient to show that $\mathcal{L}_\epsilon$ defined on $L^p(0,L)$ with the domain $W^{2,p}(0,L)$ has only trivial kernel.  Our proof is quite similar to that of Lemma 5.4 presented by Lou and Ni in \cite{LN2}.  We argue by contradiction.  Take a sequence $\{(\epsilon_i,\lambda_i)\}_{i=1}^\infty$ such that $\epsilon_i \rightarrow 0$ and $\lambda_i \rightarrow \lambda \in \Big(\frac{a_2}{b_2}+\delta, \frac{(a_2+c_2)^2}{4b_2c_2}-\delta \Big)$ as $i\rightarrow \infty$.  Without loss of our generality, we assume that there exists $\Phi_i\in W^{2,p}(0,L)$ satisfying
\begin{equation}\label{54}
\left\{
\begin{array}{ll}
\epsilon_i \frac{d^2\Phi_{i}}{dx^2}+f_v(\lambda_i,V_{\epsilon_i}(\lambda_i,x))\Phi_i=0,x\in(0,L),\\
\Phi_i'(0)=\Phi_i'(L)=0,\\
\sup_{x\in(0,L)}   \Phi_i(x)  =1.
\end{array}
\right.
\end{equation}
Let
\[\tilde{\Phi}_i(z)=\Phi_i(x_0+\sqrt{\epsilon_i}z),~\tilde{V}_{\epsilon_i}(\lambda_i,z)=V_{\epsilon_i}(\lambda_i,x_0+\sqrt{\epsilon_i}z),\]
for all $z\in\big(x_0-\frac{1}{\sqrt{\epsilon_i}},x_0+\frac{1}{\sqrt{\epsilon_i}}\big)$, $i=1,2,$..., then
\[\frac{d^2\tilde\Phi_{i}}{dz^2}+f_v(\lambda_{i},\tilde V_{\epsilon_i}(\lambda_{i},z))\tilde\Phi_{i}=0,~z\in \Big(x_0-\frac{1}{\sqrt{\epsilon_i}},x_0+\frac{1}{\sqrt{\epsilon_i}}\Big).\]
It is easy to know that both $f_v(\lambda_i, \tilde V_{\epsilon_i})$ and $\tilde\Phi_i$ are bounded in $\big(x_0-\frac{1}{\sqrt{\epsilon_i}},x_0+\frac{1}{\sqrt{\epsilon_i}}\big)$, therefore we have from the elliptic regularity and a diagonal argument that, after passing to a subsequence if necessary as $i \rightarrow \infty$, $\tilde \Phi_i$ converges to some $\tilde \Phi_0$ in $C^1(\mathbb{R}_c)$ for any compact subset $\mathbb{R}_c$ of $\mathbb{R}$; moreover, $\tilde\Phi_0$ is a $C^\infty$--smooth and bounded function and it satisfies
\begin{equation}\label{55}
\frac{d^2\tilde\Phi_0 }{dz^2}+f_v(\lambda,V_0(\lambda,z))\tilde\Phi_0=0,z\in \mathbb{R},
\end{equation}
where $V_0(\lambda,z)$ is the unique solution of (\ref{44}).

Assume that $\Phi_i(x_i)=1$ for $x_i\in[0,L]$, then we have that $f_v(\lambda_i,V_{\epsilon_i}(\lambda,x_i))\geq 0$ according to the Maximum Principle.  It follows from the decaying properties of $V_\epsilon$ that $\vert z_i\vert =\vert \frac{x_i-x_0}{\sqrt{\epsilon_i}}\vert \leq C_0 $ for some $C_0$ independent of $\epsilon_i$.  Actually, if not and we assume that there exists a sequence $z_i \rightarrow \pm \infty$ as $\epsilon_i \rightarrow 0$, then it easily implies that
$\tilde V_{\epsilon_i}\big(\lambda, z_i \big) \rightarrow \bar v_2$ and $0$ respectively.  On the other hand, we have that $f_v(\lambda, \bar v_2)<0$ and $f_v(\lambda, 0)<0$, hence $f_v(\lambda,V_{\epsilon_i} (\lambda, x_i ))=f_v(\lambda,\tilde V_{\epsilon_i} (\lambda, z_i ))<0$ for all $\epsilon_i$ small and we reach a contradiction.  Therefore we have that $z_i$ is bounded for all $\epsilon_i$ small as claimed and we can always find some $z_0\in \mathbb{R}$ such that $z_i\rightarrow z_0$ and
\[\tilde\Phi_0(z_0)=\sup_{z\in\mathbb{R}}\tilde\Phi_0(z)=1,~\tilde\Phi_0'(z_0)=0.\]

On the other hand, we differentiate equation (\ref{44}) with respect to $z$ and obtain that
\begin{equation}\label{56}
\frac{d^2 V'_0}{dz^2}+f_v(\lambda,V_0(\lambda,z))V'_0=0,
\end{equation}
where $V'_0=\frac{dV_0}{dz}$.  Multiplying (\ref{56}) by $\tilde \Phi_0$ and (\ref{57}) by $V'_0$ and then integrating them over $(-\infty,z_0)$ by parts, we obtain that
\[0=\int_{-\infty}^{z_0} \tilde\Phi_0''V'_0- (V''_0)'\tilde\Phi_0 dz=\tilde\Phi_0'(z)V'_0(z)\big \vert_{-\infty}^{z_0}-\tilde\Phi_0(z)V''_0(z)\big \vert_{-\infty}^{z_0},\]
then we can easily show that $\tilde\Phi_0(z_0)=0$ and this is a contradiction.  Therefore, we have prove in invertibility of $\mathcal{L}_\epsilon$ and we denote it inverse by $\mathcal{L}^{-1}_\epsilon$.

To show that $\mathcal{L}^{-1}_\epsilon$ is uniformly bounded for all $p\in[1,\infty]$, it suffices to prove it for $p=2$ thanks to the Marcinkiewicz interpolation Theorem.
We consider the following eigen-value problem
\begin{equation}\label{57}
\left\{
\begin{array}{ll}
\mathcal{L}_\epsilon \varphi_{i,\epsilon}=\mu_{i,\epsilon} \varphi_{i,\epsilon},~x\in(0,L),\\
\varphi_{i,\epsilon}'(0)=\varphi_{i,\epsilon}'(L)=0,\\
\sup_{x\in(0,L)}   \varphi_{i,\epsilon}(x)  =1.
\end{array}
\right.
\end{equation}
By applying the same analysis as above, we can show that for each $\lambda\in\Big(\frac{a_2}{b_2}, \frac{(a_2+c_2)^2}{4b_2c_2} \Big)$, there exists a constant $C(\lambda)>0$ independent of $\epsilon$ such that $\mu_{i,\epsilon}\geq C(\lambda)$ for all $\epsilon$ sufficiently small.  Therefore
\[\Vert \mathcal{L}_\epsilon^{-1} g\Vert_{2} =\big\Vert \sum_{j=1}^\infty \frac{<g,\varphi_{i,\epsilon}>}{\mu_{i,\epsilon}} \varphi_{i,\epsilon} \big\Vert_{2}\leq \frac{1}{C(\lambda)} \Vert g\Vert_{2},\]
where $<\cdot,\cdot>$ denotes the inner product in $L^2(0,L)$.  This finishes the proof of Lemma \ref{lem43}.
\end{proof}

\begin{proposition} Let $x_0 \in (0,L)$ be arbitrary.  Suppose that $\lambda \in \Big(\frac{a_2}{b_2}+\delta, \frac{(a_2+c_2)^2}{4b_2c_2}-\delta \Big)$ for $\delta>0$ small.  Then there exists a small $\epsilon_4=\epsilon_4(\delta)>0$ such that for all $\epsilon\in (0,\epsilon_4)$, (\ref{41}) has a family of solutions $v_\epsilon(\lambda,x)$ such that
\[\sup_{x\in(0,L)} \vert v_\epsilon(\lambda,x)-V_\epsilon(\lambda,x) \vert \leq C_4\sqrt \epsilon,\] where $C_4$ is a positive constant independent of $\epsilon$ and $V_\epsilon(\lambda,x)$ is given by (\ref{46a}).  In particular,
\begin{equation}\label{58}
\lim_{\epsilon \rightarrow 0^+} v_\epsilon(\lambda,x)=
\left\{
\begin{array}{ll}
0, \text{ compact uniformly on } [0,x_0)  ,    \\
\bar v_2(\lambda)/2,~x=x_0,\\
\bar v_2(\lambda),\text{ compact uniformly on } (x_0,L].
\end{array}
\right.
\end{equation}
\end{proposition}

\begin{proof}
We shall establish the existence of $v_\epsilon$ in the form of $v_\epsilon=V_\epsilon+\sqrt \epsilon \Psi$, where $\Psi$ satisfies (\ref{47}).  Then it is equivalent to show the existence of smooth functions $\Psi$.  To this end, we want to apply the contraction mapping theorem to the Banach space $C([0,L])$.  For any $\Psi\in C([0,L])$, we define
\begin{equation}\label{59}
\mathcal{T}_\epsilon[\Psi]=-\mathcal{L}^{-1}_\epsilon(\mathcal{G}_\epsilon+\mathcal{H}[\Psi]).
\end{equation}
Then $\mathcal{T}$ is mapping from $C([0,L])$ to $C([0,L])$ by elliptic regularity.  Moreover, we set
\[\mathcal{B}=\{\Psi\in C([0,L]) \vert  \Vert \Psi \Vert_\infty \leq R_0 \},\]
where $R_0\geq 2C_1C_3$.  By Lemma \ref{lem41} and \ref{lem43}, we have $\Vert \mathcal{L}^{-1} \mathcal{G}_\epsilon \Vert_\infty\leq C_1C_3$.  Therefore, it follows from (\ref{52}) and Lemma \ref{lem42} that, for any $\Psi\in \mathcal{B}$,
\[\Vert \mathcal{T}_\epsilon[\Psi]\Vert_\infty \leq C_1C_3+C_2C_3\sqrt{\epsilon} R_0^2\leq 2C_1C_3\leq R_0,\]
provided that $\epsilon$ is small.  Moreover, it follows from (\ref{53}) and simple calculation that, for any $\Psi_1$ and $\Psi_2$ in $\mathcal{B}$,
\[\Vert \mathcal{T}_\epsilon[\Psi_1]-\mathcal{T}_\epsilon[\Psi_2]\Vert_\infty \leq \frac{1}{2} \Vert \Psi_1-\Psi_2 \Vert_\infty.\]
 if $\epsilon$ is sufficiently small.  Hence $\mathcal{T}_\epsilon$ is a contraction mapping from $\mathcal{B}$ to $\mathcal{B}$ and it follows from the contraction mapping theory that $\mathcal{T}_\epsilon$ has a fixed point $\Psi_\epsilon$ in $\mathcal{B}$, if $\epsilon$ is sufficiently small.  Therefore, $v_\epsilon$ constructed above is a smooth solution of (\ref{41}).  Finally, it is easy to verify that $v_\epsilon$ satisfies (\ref{58}) and this completes the proof of Proposition 2.

\end{proof}

We proceed to employ the solution $v_\epsilon(\lambda,x)$ of (\ref{41}) obtained in Proposition 2 to constructed solutions of (\ref{40}).  Therefore, we want to show that there exists $\lambda=\lambda_\epsilon$ and $(v_\epsilon(\lambda_\epsilon,x), \lambda_\epsilon)$ such that the integral condition in (\ref{40}) is satisfied.


Now we are ready to present another main result of this paper.

\begin{theorem}\label{thm44}
Assume that $a_2-c_2>\frac{2a_1c_2}{c_1}$ and $b_1 \rightarrow 0 \text{ as }\epsilon \rightarrow 0$.  Denote
\[x_1 = \max \Big\{0,\frac{(a_2-c_2)c_1-2a_1c_2}{(a_1+c_1)(a_2-c_2)}L\Big\},~x_2 =\frac{(a_2-c_2)c_1-a_1c_2}{(a_1+c_1)(a_2-c_2)}L.\]
Then there exists $\epsilon_0>0$ small such that for each $x_0\in(x_1,x_2)$ and $\epsilon \in (0,\epsilon_0)$, system (\ref{1}) admits positive solutions $(v_\epsilon(\lambda_\epsilon,x),\lambda_\epsilon)$ such that
\begin{equation}\label{60}
\lim_{\epsilon \rightarrow 0^+} v_\epsilon(\lambda_\epsilon,x)=
\left\{
\begin{array}{ll}
0,&\text{ compact uniformly on } [0,x_0),\\
\bar v_2(\lambda_0)/2,& x=x_0,\\
\bar v_2(\lambda_0), &\text{ compact uniformly on } (x_0,L],
\end{array}
\right.
\end{equation}
where $\bar v_2(\lambda_0)=\frac{a_1L}{c_1-(a_1+c_1)x_0}\in (\frac{a_2-c_2}{2c_2},\frac{a_2-c_2}{c_2})$, and
\begin{equation}\label{61}
\lim_{\epsilon \rightarrow 0^+} \lambda_\epsilon = \bar \lambda_0=\frac{(a_2-c_2\bar v_0)(1+\bar v_0)}{b_2} \in \Big(\frac{a_2}{b_2},\frac{(a_2+c_2)^2}{4b_2c_2} \Big).
\end{equation}
\end{theorem}

\begin{remark}
We note that the assumption $a_2-c_2>\frac{2a_1c_2}{c_1}$ is exactly the same as (\ref{12}) when $b_1=0$ and this condition is required to guarantee the existence of small amplitude bifurcating solutions.  In particular, we have that $x_1=0$ if $c_2< a_2\leq \frac{2a_1}{C}+c_2 $ and $x_1=\frac{(a_2-c_2)c_1-2a_1c_2}{(a_1+c_1)(a_2-c_2)}L $ if $a_2 > \frac{2a_1}{C}+c_2$; moreover $x_2<L$ for all $a_2>c_2$.  Similar as the stability analysis in Section 3, the limit assumption on $b_1$ is only made for the sake of mathematical simplicity.
\end{remark}

\begin{proof}
We shall apply the Implicit Function Theorem for the proof.  To this end, we define for all $\epsilon\in(-\delta,\delta)$ for $\delta$ being sufficiently small,
\begin{equation}\label{62}
\mathcal{I}(\epsilon,\lambda)=\int_0^L \frac{a_1-c_1v_\epsilon(\lambda,x)}{1+v_\epsilon(\lambda,x)}dx- \int_0^L \frac{b_1 \lambda}{(1+v_\epsilon(\lambda,x))^2} dx,
\end{equation}
where $\lambda \in \big(\bar \lambda_0-\delta , \bar \lambda_0+\delta \big)$ and $\bar \lambda_0$ is a positive constant to be determined.  For $\epsilon\leq 0$, we set $v_\epsilon(\lambda,x)=0$ if $x\in [0,x_0)$ and $v_\epsilon(\lambda,x)=\bar v_2(\lambda)$ if $x\in(x_0,L]$.  Then we have that
\[\mathcal{I}(\epsilon,\lambda)\equiv  a_1 x_0+\frac{L-x_0}{1+\bar v_2(\lambda)}\big(a_1-c_1\bar v_2(\lambda)\big),~\forall \epsilon \leq 0\]
On the other hand, for $\epsilon>0$, we have from (\ref{62}) that $\mathcal{I}(\epsilon,\lambda)$ is a smooth function of $\lambda$ and
\begin{equation}\label{63}
\frac{\partial \mathcal{I}(\epsilon,\lambda)}{\partial \lambda }=\int_0^L  \frac{2b_1 \lambda-(a_1+c_1)(1+v_\epsilon(\lambda,x))  }{(1+v_\epsilon(\lambda,x))^3} \frac{\partial v}{\partial \lambda}dx-\int_0^L \frac{b_1 }{(1+v_\epsilon(\lambda,x))^2} dx;
\end{equation}
moreover, we have from Proposition 2 that, $\lim_{\epsilon \rightarrow 0^+} \frac{\partial v}{\partial \lambda} \equiv 0$ if $x\in[0,x_0)$ and $\lim_{\epsilon \rightarrow 0^+} \frac{\partial v}{\partial \lambda} \equiv -\frac{b_2}{\sqrt{(a_2+c_2)^2-4b_2c_2\lambda}}$ if $x\in(x_0,L]$, where the convergence is pointwise in both cases.  By the Lebesgue Dominated Convergence Theorem, we see that $\lim_{\epsilon \rightarrow 0^+} \frac{\partial \mathcal{I}(\epsilon,\lambda)}{\partial \lambda }=\frac{(a_1+c_1)b_2(L-x_0)}{(1+\bar v_2)^2\sqrt{(a_2+c_2)^2-4b_2c_2\lambda}} \neq0$ for all $\lambda \neq \frac{(a_2+c_2)^2}{4b_2c_2}$.  Hence $\frac{\partial \mathcal{I}(\epsilon,\lambda)}{\partial \lambda }$ is continuous in a neighborhood of $(0, \bar \lambda_0)$ for all $\bar \lambda_0 \in \big(\frac{a_2}{b_2},\frac{(a_2+b_2)^2}{4b_2c_2} \big)$.  Therefore, according to the Implicit Function Theorem, in a small neighbourhood of $(\epsilon,\lambda)=(0,\lambda_0)$, there exists $\lambda=\lambda_\epsilon$ and $(v_\epsilon(\lambda_\epsilon,x),\lambda_\epsilon)$ is a solution to system (\ref{1}) such that $\lambda_\epsilon \rightarrow \bar \lambda_0$ as $\epsilon \rightarrow 0^+$.

To determine the values of $\bar v_0$ and $\bar \lambda_0$, we send $\epsilon$ to zero and conclude from (\ref{1}) and the Lebesgue Dominated Convergence Theorem that,
\begin{equation}
\left\{
\begin{array}{ll}
a_2-\frac{b_2 \bar \lambda_0}{1+\bar v_2(\bar \lambda_0)}-c_2\bar v_2(\bar \lambda_0)=0,\\
a_1x_0+\big(\frac{a_1-c_1\bar v_2(\bar\lambda_0)}{1+\bar v_2(\bar\lambda_0)} \big)(L-x_0)=0,
\end{array}
\right.
\end{equation}
then it follows from straightforward calculations that $\bar v_2(\bar \lambda_0)=\frac{a_1 L}{c_1L -(a_1+c_1)x_0}$ and $\bar \lambda_0=\frac{(a_2-c_2\bar v_2(\bar \lambda_0))(1+\bar v_2(\bar \lambda_0))}{b_2}$.  Moreover, since $\lambda_0 \in \big(\frac{a_2}{b_2},\frac{(a_2+c_2)^2}{4b_2c_2} \big)$, it is equivalent to have that $\bar v_2(\bar \lambda_0) \in (\frac{a_2-c_2}{2c_2},\frac{a_2-c_2}{c_2})$, which implies that $x_0\in(x_1,x_2)$ as in Theorem \ref{thm44} through straightforward calculations.  This verifies (\ref{60}) and (\ref{61}) and completes the proof of Theorem \ref{thm44}.
\end{proof}

\section{Conclusion and Discussion}
In this paper, we carry out local and global bifurcation analysis in (\ref{1})and establish the nonconstant positive solutions $v_\epsilon(\lambda_\epsilon,x)$ to this nonlinear problem.  It is shown that the bifurcating solutions exist for all $\epsilon>0$ being small--see (\ref{12}).  Though it might be well-known to some people and it may hold even for general reaction-diffusion systems, we show that all the local branches must be of pitch-fork type.  For the simplicity of calculations, we assume that $b_1=0$ and the stability of these bifurcating solutions are then determined.  In particular, we have that the bifurcating solutions are always unstable as long as $\epsilon$ is sufficiently small.  Finally, we constructed positive solutions to (\ref{1}) that have a single transition layer, where again we have assumed that $b_1=0$ for the sake of mathematical simplicity.  Our results complement \cite{LN2} on the structures of the nonconstant positive steady states of (\ref{1}) and help to improve our understandings about the original SKT competition system (\ref{2}).

We want to note that, though the assumption $b_1=0$ in Section 3 and Section 4 is made for the sake of mathematical simplicity, it is interesting question to answer whether or not (\ref{1}) admits solutions for all $B<A<C$ or $C<A<B$.  It is also an interesting and important question to probe on the global structure of all the bifurcation branches.  It is proved in \cite{SW} that the continuum of each bifurcation branch must satisfy one of three alternatives, and new techniques need to be developed in order to rule out or establish the compact global branches.  Moreover, more information on the limiting behavior of $v_\epsilon$ not only as $\epsilon$ approaches to zero, but some positive critical value which may also generates nontrivial patterns.  See \cite{LNY} for the work on a similar system.  The stability of the transition-layer solutions is yet another important and mathematically challenging problem that worths attention.  To this end, one needs to construct approximating solutions to (\ref{1}) of at least $\epsilon$--order.  Therefore, more information is required on the operator $\mathcal{L}_\epsilon$, for example, the limiting behavior of its second eigenvalue.

Our mathematical results are coherent with the phenomenon of competition induced species segregation.  We see from the limiting profile analysis of (\ref{2}) in \cite{LN2} that $u(1+v)$ converges to the positive constant $\lambda$ as $\rho_{12}\rightarrow \infty$ provided that $\rho_{12}$ and $d_1$ are comparable.  Then the existence of the transition layer in $v$ implies that $u=\frac{\lambda}{1+v}$ must be in the form of an inverted transition layer for $\epsilon$ being small.  These transition-layers solution can be useful in mathematical modelings of species segregation.  Therefore, the species segregation is formed through a mechanism cooperated by the diffusion rates $d_1$, $d_2$ and the cross-diffusion pressure $\rho_{12}$.  Eventually, the structure of $v_\epsilon(x)$ in (\ref{1}) provides essential understandings about the original system (\ref{2}).

\medskip
\medskip

\end{document}